\documentclass[11pt]{article}
\usepackage{enumerate}
\usepackage{amssymb,a4wide,latexsym,makeidx,epsfig}
\usepackage{amsthm}
\usepackage{amsmath}
\usepackage{lipsum} % for filler text
\usepackage{enumerate}
\usepackage{mathrsfs}
\usepackage{xcolor}
\usepackage{tikz}%%%%draw graph%%%PDFLaTeX or PDFTeXify
\usepackage{geometry}
\usepackage{appendix}
\usepackage{ulem}
\allowdisplaybreaks
\newtheorem{theorem}{Theorem}[section]
\newtheorem{remark}[theorem]{Remark}

\newtheorem{lemma}[theorem]{Lemma}

\newtheorem{proposition}[theorem]{Proposition}
\newtheorem{corollary}[theorem]{Corollary}

\newtheorem{conjecture}[theorem]{Conjecture}
\newtheorem{claim}[theorem]{Claim}

\begin{document}
\textwidth 150mm \textheight 225mm
\title{The Erd\H{o}s-Gy\'{a}rf\'{a}s function with respect to Gallai-colorings
\thanks{Supported by the National Natural Science Foundation of China (No. 11871398)
and China Scholarship Council (No. 201906290174).}
}
\author{{Xihe Li$^{1,2}$, Hajo Broersma$^{2,}$\thanks{Corresponding author. Orcid id: https://orcid.org/0000-0002-4678-3210}, Ligong Wang$^{1}$}\\
{\small $^{1}$ School of Mathematics and Statistics,}\\ {\small Northwestern Polytechnical University, Xi'an, Shaanxi 710129, PR China}\\
{\small $^{2}$ Faculty of Electrical Engineering, Mathematics and Computer Science,}\\ {\small University of Twente, P.O. Box 217, 7500 AE Enschede, The Netherlands}\\
{\small E-mail: lxhdhr@163.com; h.j.broersma@utwente.nl; lgwangmath@163.com}}
\date{}
\maketitle
\begin{center}
\begin{minipage}{120mm}
\vskip 0.3cm
\begin{center}
{\small {\bf Abstract}}
\end{center}
{\small For fixed $p$ and $q$, an edge-coloring of the complete graph $K_n$ is said to be a $(p, q)$-coloring if every $K_p$ receives at least $q$ distinct colors. The function $f(n, p, q)$ is the minimum number of colors needed for $K_n$ to have a $(p, q)$-coloring. This function was introduced about 45 years ago, but was studied systematically by Erd\H{o}s and Gy\'{a}rf\'{a}s in 1997, and is now known as the Erd\H{o}s-Gy\'{a}rf\'{a}s function. In this paper, we study $f(n, p, q)$ with respect to Gallai-colorings, where a Gallai-coloring is an edge-coloring of $K_n$ without rainbow triangles. Combining the two concepts, we consider the function $g(n, p, q)$ that is the minimum number of colors needed for a Gallai-$(p, q)$-coloring of $K_n$. Using the anti-Ramsey number for $K_3$, we have that $g(n, p, q)$ is nontrivial only for $2\leq q\leq p-1$. We give a general lower bound for this function and we study how this function falls off from being equal to $n-1$ when $q=p-1$ and $p\geq 4$ to being $\Theta(\log n)$ when $q = 2$. In particular, for appropriate $p$ and $n$, we prove that $g=n-c$ when $q=p-c$ and $c\in \{1,2\}$, $g$ is at most a fractional power of $n$ when $q=\lfloor\sqrt{p-1}\rfloor$, and $g$ is logarithmic in $n$ when $2\leq q\leq \lfloor\log_2 (p-1)\rfloor+1$.

 }
\end{minipage}
\end{center}

\section{Introduction}
\label{sec:ch-introduction}

Let $p$ and $q$ be positive integers with $2\leq q\leq \binom{p}{2}$. An edge-coloring of the complete graph $K_n$ is said to be a {\it $(p, q)$-coloring} if every $K_p$ receives at least $q$ distinct colors. The function $f(n, p, q)$ is defined to be the minimum number of colors that are needed for $K_n$ to have a $(p, q)$-coloring. This function was first introduced by Erd\H{o}s and Shelah \cite{Erd1,Erd2}, but Erd\H{o}s and Gy\'{a}rf\'{a}s \cite{ErGy} were the first to study it in depth; it is now known as the Erd\H{o}s-Gy\'{a}rf\'{a}s function. This function generalizes the multicolored Ramsey number, since determining $f(n, p, 2)$ is equivalent to determining the Ramsey number of $K_p$.

In \cite{ErGy}, Erd\H{o}s and Gy\'{a}rf\'{a}s determined various thresholds for $f(n, p, q)$. In particular, they proved that $q={p \choose 2}-p+3$ is the smallest value of $q$ such that $f(n, p, q)$ is linear in $n$, and $q={p \choose 2}-\left\lfloor\frac{p}{2}\right\rfloor+2$ is the smallest value of $q$ such that $f(n, p, q)$ is quadratic in $n$. In \cite{CFLS2}, Conlon et al. proved that $q=p$ is the smallest value of $q$ such that $f(n, p, q)$ is polynomial in $n$.

The exact value of the Erd\H{o}s-Gy\'{a}rf\'{a}s function is very difficult to determine, even for some small values of $p$ and $q$. For example, the best known lower bound for $f(n, 4, 3)$ is $O(\log n)$ \cite{FoSu}, while the best until now upper bound is $e^{O(\sqrt{\log n})}$ \cite{Mub}. There is clearly a large gap between the lower and upper bound. On the other hand, some special cases of this function are closely related to other interesting problems. For example, $f(n,9,34)$ relates to a Tur\'{a}n type hypergraph problem posed by Brown, Erd\H{o}s and S\'{o}s \cite{BrES,ErGy}, $f(n,5,9)$ relates to sets containing no 3-term arithmetic progression \cite{Axe}, and $f(n,3,3)$ and $f(n,5,9)$ relate to some problems on properly colored complete graphs \cite{ErGy,Ros}. For more information on this function, we refer to \cite{AxFM,CFLS1,CFLS2,FiPS,PoSh,SaSe1,SaSe2} and Section 3.5.1 of \cite{CoFS}.

A graph with an edge-coloring is called {\it rainbow} if all its edges are colored differently. A {\it Gallai-$k$-coloring} is a $k$-edge-coloring of a complete graph $K_n$ without rainbow triangles (that is, every triangle receives at most two colors). In this paper, we investigate the Erd\H{o}s-Gy\'{a}rf\'{a}s function within the framework of Gallai-colorings. A Gallai-coloring of the complete graph $K_n$ is said to be a {\it Gallai-$(p, q)$-coloring} if every $K_p$ receives at least $q$ distinct colors. We define $g(n, p, q)$ to be the minimum number of colors that are needed for $K_n$ to have a Gallai-$(p, q)$-coloring. Clearly, we have $f(n,p,q)\leq g(n,p,q)$ if both functions are defined for these values of $n$, $p$ and $q$.

For studying $g(n, p, q)$ it is convenient to introduce the following function. For $1\leq q\leq \binom{p}{2}$, let $g^k_q(p)$ be the smallest positive integer $n$ such that every Gallai-$k$-coloring of $K_n$ contains a copy of $K_p$ receiving at most $q$ distinct colors. Restated, $g^k_q(p)-1$ is the largest positive integer $n$ such that there is a Gallai-$k$-coloring of $K_n$ in which every $K_p$ receives at least $q+1$ distinct colors, i.e., such that $g(n, p, q+1)\leq k$. Throughout the remainder of the paper, we concentrate on the function $g^k_q(p)$ and derive upper and lower bounds and some exact values for this function. We reflect on what these results on $g^k_q(p)$ imply for the function $g(n, p, q)$ in Section~\ref{sec:ch-remark}. It is worth noting that Erd\H{o}s introduced an analogue of the function $g^k_q(p)$ when he posed the problem on $f(n, p, q)$ in his original paper~\cite{Erd1}.

We first point out that $g^k_q(p)$ is nontrivial only for $1\leq q\leq p-2$ (equivalently, $g(n, p, q)$ is nontrivial only for $2\leq q\leq p-1$). When $q\geq p-1$, we can deduce $g^k_q(p)$ using the following anti-Ramsey result.

\noindent\begin{theorem}\label{th:antiR} {\normalfont (\cite{ErSS,GySi})}
At most $p-1$ colors can be used in any Gallai-coloring of $K_p$.
\end{theorem}

\noindent\begin{corollary}\label{co:q=p-1}
For integers $k\geq 1$, $p\geq 3$ and $q\geq p-1$, there is no Gallai-$k$-coloring of $K_n$ in which every $K_p$ receives at least $q+1$ distinct colors. Thus $g^k_q(p)=p$ for $q\geq p-1$.
\end{corollary}

Moreover, if $k< q$, then it is obvious that $g^k_q(p)=p$. In the sequel, we will always assume that $k\geq q$ and $1\leq q\leq p-2$ when we consider $g^k_q(p)$. Note that we have the following inequalities:
$$g^k_q(p)\leq g^{k+1}_q(p),\ \ g^k_{q+1}(p)\leq g^k_q(p)\ \ {\rm and\ \ } g^{k+1}_{q+1}(p)\leq g^k_q(p),$$
as we now explain. The first two inequalities hold by the definition of $g^k_q(p)$. For the third inequality, let $n_0=g^{k+1}_{q+1}(p)-1$. Then there exists a Gallai-($k+1$)-coloring $G$ of $K_{n_0}$ in which every $K_p$ receives at least $q+2$ colors. Let $G'$ be an edge-coloring of $K_{n_0}$ obtained from $G$ by unifying colors $k$ and $k+1$. Clearly $G'$ is a Gallai-$k$-coloring in which every $K_p$ receives at least $q+1$ colors. Thus $g^k_q(p)\geq n_0+1=g^{k+1}_{q+1}(p)$.

In \cite{FoGP}, Fox, Grinshpun and Pach proved the following asymptotic result. Note that for $k=3$ and $q=2$, this result is a special case of the multicolor generalization of the well-known Erd\H{o}s-Hajnal conjecture.

\noindent\begin{theorem}\label{th:general} {\normalfont (\cite{FoGP})}
Let $k$ and $q$ be fixed positive integers with $q\leq k$. Every Gallai-$k$-coloring of $K_n$ contains a set of order $\Omega (n^{\binom{q}{2}/\binom{k}{2}}\log^{c_{k,q}}_2 n)$ which uses at most $q$ colors, where $c_{k,q}$ is only depending on $k$ and $q$. Moreover, this bound is tight apart from the constant factor.
\end{theorem}

It is worth noticing that the problem studied by Fox, Grinshpun and Pach is to find the largest subgraph $K_p$ using at most $q$ colors in every Gallai-$k$-coloring of $K_n$, for fixed $k$ and $q$, when $n$ is sufficiently large. But in this paper, we mainly focus on the problem to determine the smallest $n$ such that there is a $K_p$ using at most $q$ colors in every Gallai-$k$-coloring of $K_n$, for fixed $p$ and $q$, when $k\in [1, +\infty)$ (or $k\to \infty$). Therefore, the above theorem cannot give us much support, since it requires that $n$ is sufficiently large, in fact, $$n\geq n_0= 2^{2^{2^{2^{8k^2}}}}.$$ But we can prove an upper bound of $2^{\frac{2k(p-2)}{q}+1}$ on $g^k_{q}(p)$ (see Theorem \ref{th:q} below). If $2^{\frac{2k(p-2)}{q}+1}\geq n_0$, then $k=o(p)$, which implies that for fixed $p$ and $q$, only $o(p)$ $g^k_{q}(p)$'s can be bounded using the above theorem.

\noindent\begin{theorem}\label{th:q}
For integers $p,q,k$ with $p\geq 3$, $1\leq q\leq p-2$ and $k\geq q$, we have $g^k_{q}(p)\leq 2^{\frac{2k(p-2)}{q}+1}$.
\end{theorem}

We postpone all proofs of our results to later sections. Note that Theorem \ref{th:q} implies that $g(n,p,q)>\frac{q-1}{2(p-2)}(\log_2 n -1)$, where $p\geq 3$, $2\leq q\leq p-1$ and $n\geq 2^{2p-3}$. In \cite{ErGy}, Erd\H{o}s and Gy\'{a}rf\'{a}s obtained an upper bound for $f(n, p, q)$ using the Lov\'{a}sz Local Lemma. However, it seems difficult to determine a nontrivial general upper bound for $g(n, p, q)$ (or, equivalently, lower bound for $g^k_q(p)$). Although we can prove some nontrivial results (see, for example Proposition \ref{prop:sp} below) using the Lov\'{a}sz Local Lemma, it cannot help us much in determining an upper bound for $g(n, p, q)$. A graph with an edge-coloring is called $q$-{\it colored} if its edges are colored with at most $q$ distinct colors.

\noindent\begin{proposition}\label{prop:sp} For fixed integers $s,$ $q,$ $k,$ and appropriately large integer $p$ with $s\geq 4$ and $k\geq \max \left\{\binom{s}{2}, 2q+1\right\}$, there exists a $k$-edge-coloring of $K_n$ with
$$n=\left(\frac{(s-2)pL^{1/\left(1-\binom{s}{2}\right)}}{(c+o(1))\left(\binom{s}{2}-2.1\right) \ln \left(pL^{1/\left(1-\binom{s}{2}\right)}\right)}\right)^{\left(\binom{s}{2}-2.1\right)/(s-2)}$$
such that there is neither a rainbow $K_s$ nor a $q$-colored $K_p$, where $c$ is a constant and $L=\binom{s}{2}(k-1)^{2-\binom{s}{2}}(k-2)\cdots\left(k-\binom{s}{2}+1\right)$.
\end{proposition}

When $q=1$, $g^k_1(p)$ is the smallest positive integer $n$ such that every Gallai-$k$-coloring of $K_n$ contains a monochromatic copy of $K_p$. Fox, Grinshpun and Pach \cite{FoGP} posed the following conjecture, which was verified independently by Chung and Graham \cite{ChGr} and Gy\'{a}rf\'{a}s et al. \cite{GSSS} for $p=3$, and by Liu et al. \cite{LMSSS} for $p=4$, using the language of Gallai-Ramsey numbers.

\noindent\begin{conjecture}\label{conj:Kt} {\normalfont (\cite{FoGP})}
For integers $k\geq 3$ and $p\geq 3$,
$$g^k_1(p)=
\left\{
   \begin{aligned}
    &(R_2(K_p)-1)^{k/2}+1, & & \mbox{if $k$ is even},\\
    &(p-1)\cdot(R_2(K_p)-1)^{(k-1)/2}+1, & & \mbox{if $k$ is odd},
   \end{aligned}
   \right.$$
where $R_2(K_p)$ is the 2-colored Ramsey number for $K_p$.
\end{conjecture}

We can slightly improve Theorem \ref{th:q} for $q=1$ by proving the following upper bound on $g^k_1(p)$.

\noindent\begin{proposition}\label{prop:Kp} For integers $k\geq 3$ and $p\geq 5$, we have $g^k_1(p)< 2^{2k(p-2)-3}$.
\end{proposition}

When $q=p-2$, we can prove the following result, thereby improving some results obtained in \cite{BeBu}.

\noindent\begin{theorem}\label{th:q=p-2}
For integers $p\geq 4$ and $k\geq p-2$, we have $g^k_{p-2}(p)=k+2$.
\end{theorem}

The above result is equivalent to $g(n,p,p-1)=n-1$, where $n\geq p\geq 4$.
Using Theorem \ref{th:q=p-2}, we can show that $g^k_q(p)$ is at least quadratic in $k$ for $q=\left\lfloor \sqrt{p-1}\right\rfloor-1$.

\noindent\begin{theorem}\label{th:q=sqrt}
For integers $p\geq 17$ and $k\geq \left\lfloor \sqrt{p-1}\right\rfloor-1$, we have $g^k_{\left\lfloor \sqrt{p-1}\right\rfloor-1}(p)\geq k^2+2k+2$.
\end{theorem}

Note that Theorem \ref{th:q=sqrt} implies that $g(n,p,\lfloor\sqrt{p-1}\rfloor)\leq \left\lceil\sqrt{n}\right\rceil-1$ for $p\geq 17$ and $n\geq \left(\left\lfloor\sqrt{p-1}\right\rfloor+1\right)^2$.
When $q=p-3$, we can prove the following result, which is equivalent to $g(n,p,p-2)=n-2$ for $n\geq p\geq 8$.

\noindent\begin{theorem}\label{th:q=p-3}
For integers $p\geq 8$ and $k\geq p-3$, we have $g^k_{p-3}(p)=k+3$.
\end{theorem}

Furthermore, we can determine the exact value of $g^k_2(5)$. Using this result, we can show that $g^k_q(p)$ is exponential in $k$ for all $1\leq q\leq\left\lfloor \log_2 (p-1)\right\rfloor$.

\noindent\begin{theorem}\label{th:p=5}
For integers $k\geq 2$, we have $g^k_2(5)=2^k+1$.
\end{theorem}

\noindent\begin{theorem}\label{th:q=log}
For integers $p\geq 5$ and $k\geq \left\lfloor \log_2 (p-1)\right\rfloor$, we have $g^k_{\left\lfloor \log_2 (p-1)\right\rfloor}(p)\geq 2^k+1$.
\end{theorem}

Note that Theorem \ref{th:p=5} is equivalent to $g(n,5,3)=\lceil\log_2 n\rceil$, where $n\geq 5$. Theorem \ref{th:q=log} implies that $g(n,p,\lfloor\log_2(p-1)\rfloor+1)\leq \lceil\log_2 n\rceil$, where $p\geq 5$ and $n\geq 2(p-1)$.

Finally, motivated by the problem introduced by Erd\H{o}s, Hajnal and Rado (see Section 18 of \cite{ErHR}) to find the minimum integer $n$ such that for any $k$-coloring of $K_n$ there is a $(k-1)$-colored $K_m$, we study $g^k_{k-1}(p)$ for $k\leq p-1$. If $p$ is sufficiently larger than $k$, then $g^k_{k-1}(p)=O((p/\log^c_2 p)^{k/(k-2)})$ by Theorem \ref{th:general}. So we will focus on the case $k/p\to 1$. By Theorems \ref{th:q=p-2} and \ref{th:q=p-3}, we have $g^k_{k-1}(p)=p+1$ for $k\in \{p-1, p-2\}$ and large enough $p$. A natural question is whether $g^k_{k-1}(p)=p+1$ for $k=p-c$, where $c$ is a constant and $p$ is large enough. The following theorem answers this question.

\noindent\begin{theorem}\label{th:q=k-1}
For integers $c$, $p$ and $k$ with $c\geq 1$, $p\geq 2(8+c)^{c+1}-1$ and $k=p-c$, we have $g^k_{k-1}(p)=p+1$.
\end{theorem}

The remainder of this paper is organized as follows. In the next section, we provide some useful results and additional terminology. In Section \ref{sec:ch-proof Kp}, we prove Theorem \ref{th:q} and Propositions \ref{prop:sp} and \ref{prop:Kp}. In Section \ref{sec:ch-proof q=p-2}, we give our proof of Theorem \ref{th:q=p-2}, and we prove Theorem \ref{th:q=sqrt} in a more general form. In Section \ref{sec:ch-proof q=p-3}, we present our proof of Theorem \ref{th:q=p-3}. In Section \ref{sec:ch-proof p=5}, we prove Theorems \ref{th:p=5} and \ref{th:q=log}. Section \ref{sec:ch-proof q=k-1} is devoted to our proof of Theorem \ref{th:q=k-1}. Finally, we will conclude the paper with some reflections on what our results for $g^k_q(p)$ imply for the function $g(n, p, q)$ in Section \ref{sec:ch-remark}. There we also present a conjecture and some open problems. In Appendix A and B, we provide our proofs of $g^4_{3}(6)=8$ and $g^5_{3}(6)=10$, respectively.

\section{Preliminaries}
\label{sec:ch-preliminaries}

We begin with some terminology and notation. Given a graph $G$, let $c$ : $E(G)\rightarrow [k]$ be a $k$-edge-coloring of $G$, where $[k] := \{1, 2, \ldots, k\}$. For an edge $e\in E(G)$, let $c(e)$ be the color used on edge $e$. For nonempty subsets $U$, $V\subset V(G)$ with $U\cap V= \emptyset$, let $E(U, V)=\{uv\in E(G)\colon\, u\in U, v\in V\}$ and $C(U, V)=\{c(e)\colon\, e\in E(U,V)\}$. If $\left|C(U, V)\right|=1$, then we use $c(U, V)$ to denote the unique color in $C(U, V)$. The subgraph of $G$ induced by $U$ is denoted by $G[U]$, and $G-U$ is shorthand for $G[V(G)\setminus U]$. If $U$ consists of a single vertex $u$, then we simply write $E(\{u\}, V)$, $C(\{u\}, V)$, $c(\{u\}, V)$ and $G-\{u\}$ as $E(u, V)$, $C(u, V)$, $c(u, V)$ and $G-u$, respectively. Let $C(G)$, $C(G[U])$ and $C(G-U)$ denote the set of colors used on $E(G)$, $E(G[U])$ and $E(G-U)$, respectively. We also use the abbreviation $C(U)$ for $C(G[U])$.
For a color $i$, the {\it subgraph induced by color $i$} is the subgraph that contains all the edges with color $i$ and the vertices that are incident with at least one edge of color $i$.

The following structural result on Gallai-colorings was first obtained by Gallai \cite{Gallai}, using the terminology of transitive orientations, and restated by Gy\'{a}rf\'{a}s and Simonyi \cite{GySi} in the language of graph theory.

\noindent\begin{theorem}\label{th:Gallai} {\normalfont (\cite{Gallai,GySi})}
In any Gallai-coloring of a complete graph, the vertex set can be partitioned into at least two nonempty parts such that there is only one color on the edges between every pair of parts, and there are at most two colors between the parts in total.
\end{theorem}

We call a vertex partition as given by Theorem \ref{th:Gallai} a {\it Gallai-partition}. Since every 2-edge-coloring of $K_n$ contains a connected monochromatic spanning subgraph, we have the following corollary.

\noindent\begin{corollary}\label{co:Gallai}
In any Gallai-coloring of a complete graph, there is a connected monochromatic spanning subgraph.
\end{corollary}

We shall also use the following simple result in our proofs.

\noindent\begin{lemma}\label{le:Gallai}
Let $G$ be a Gallai-coloring of a complete graph, $V\subset V(G)$ and $v\in V(G)\setminus V$. Then there is at most one color on the edges between $v$ and $V$ that is not used on any edge within $V$ {\rm (}that is, $\left|C(v,V)\setminus C(V)\right|\leq 1${\rm )}.
\end{lemma}

\begin{proof} Suppose that $c(vu)=1$, $c(vw)=2$ and $1, 2\notin C(V)$, where $u, w\in V$. Then we may further assume that $c(uw)=3$. Now $\{u, v, w\}$ forms a rainbow triangle, a contradiction.
\end{proof}

Finally, we introduce the Lov\'{a}sz Local Lemma. Let $(\Omega, \mathcal{F}, {\rm Pr})$ be a probability space and let $A_1, A_2, \ldots, A_n$ be events. A graph $D$ with $V(D)=\{v_1, v_2, \ldots, v_n\}$ is called a {\it dependency graph} for events $A_1, A_2, \ldots A_n$ if for every $i$, the event $A_i$ is mutually independent of all $A_j$ with $v_iv_j\notin E(D)$ and $i\neq j$, i.e., $A_i$ is independent of any Boolean function of the events in $\{A_j\colon\, v_iv_j\notin E(D), i\neq j\}$. We shall use the following form of the Local Lemma due to Spencer.

\noindent\begin{lemma}\label{le:local} {\normalfont (Lov\'{a}sz Local Lemma \cite{ErLo, Spe})}
Let $A_1, A_2, \ldots, A_n$ be events in a probability space $(\Omega, \mathcal{F}, {\rm Pr})$ with dependency graph $D$. If there exist positive real numbers $y_1, y_2, \ldots, y_n$ such that for each $i$, $y_i {\rm Pr}(A_i)<1$ and $\ln y_i>\sum_{v_iv_j\in E(D)}y_j {\rm Pr}(A_j)$, then ${\rm Pr}(\wedge^{n}_{i=1}\overline{A}_i)>0$.
\end{lemma}

\section{General upper and lower bounds}
\label{sec:ch-proof Kp}

Before proving Theorem \ref{th:q}, we first prove two lemmas. The proof ideas of Lemmas~\ref{le:K3} and \ref{le:reduce} below are from \cite{FoSu}. For an edge-colored $K_n$, a vertex $v\in V(K_n)$ and a color $i$, let $d_i(v)$ be the number of edges in color $i$ incident with $v$.

\noindent\begin{lemma}\label{le:K3}
If an edge-coloring of $K_n$ with $n\geq 4$ satisfies $d_i(v)\leq \frac{n}{4}$ for each $v\in V(K_n)$ and each color $i$, then there exists a rainbow copy of $K_3$.
\end{lemma}

\begin{proof} It suffices to show that the number of non-rainbow $K_3$'s is less than $\binom{n}{3}$. Note that for any vertex $v$ and any color $i$, there are at most $\binom{d_i(v)}{2}$ non-rainbow $K_3$'s with two edges in color $i$ incident with vertex $v$. Thus the number of non-rainbow $K_3$'s is at most $$\sum_{v\in V(K_n)}\sum_{i}\frac{d_i(v)(d_i(v)-1)}{2}\leq 4n\frac{(n/4)(n/4 -1)}{2}<\binom{n}{3},$$
where the first inequality holds since $\sum_{i}\frac{d_i(v)(d_i(v)-1)}{2}\leq 4\frac{(n/4)(n/4 -1)}{2}$ (using $0\leq d_i(v)\leq \frac{n}{4}$, $\sum_{i}d_i(v)= n-1$, and noting that the function $f(x)=\frac{x(x-1)}{2}$ is convex with $f(x)\geq f(1)=0$ for any $x\geq 1$).
\end{proof}

Let $[k]=\{1, 2, \ldots, k\}$ be a set of colors and $t_q=\sum^{q}_{i=1} \binom{k}{i}$. Let $\mathcal{I}=\{I\subseteq [k] : 1\leq \left|I\right|\leq q\} =\{I_1, I_2, \ldots, I_{t_q}\}$. Then we define $g^k_q(p_1, p_2, \ldots, p_{t_q})$ to be the smallest positive integer $n$ such that every Gallai-$k$-coloring of $K_n$ contains a copy of $K_{p_i}$ all edges of which have colors from one set $I_i$ for some $i$.

\noindent\begin{lemma}\label{le:reduce} We have
$$g^k_q(p_1, p_2, \ldots, p_{t_q})\leq 4\cdot \max_{1\leq i\leq k} g^k_q\left(p^{(i)}_1, p^{(i)}_2, \ldots, p^{(i)}_{t_q}\right),$$
where $p^{(i)}_j=p_j-1$ if $i\in I_j$, and $p^{(i)}_j=p_j$ otherwise.
\end{lemma}

\begin{proof} Let $n\geq 4\cdot \max_{1\leq i\leq k} g^k_q\left(p^{(i)}_1, p^{(i)}_2, \ldots, p^{(i)}_{t_q}\right)$. By Lemma \ref{le:K3}, for every Gallai-coloring of $K_n$, there exists a vertex $v$ and a color $\ell$ with $d_{\ell}(v)> \frac{n}{4}$. Let $N_{\ell}(v)=\{u : c(uv)=\ell\}$. Then $\left|N_{\ell}(v)\right|> g^k_q\left(p^{(\ell)}_1, p^{(\ell)}_2, \ldots, p^{(\ell)}_{t_q}\right)$. In this case there is a copy of $K_{p_i}$ all edges of which have colors from one set $I_i$ for some $i$. This proves the statement of the lemma.
\end{proof}

Now we have all ingredients to present our proofs of Theorem~\ref{th:q} and Proposition~\ref{prop:Kp}.

\begin{proof}[Proof of Theorem~\ref{th:q}] Note that $g^k_q(p)=g^k_q(p, p, \ldots, p)$. We can repeatedly apply Lemma \ref{le:reduce} until in some step we get $g^k_q(p_1, p_2, \ldots, p_{t_q})\leq 2$. In each step, we have $g^k_q(p_1, p_2, \ldots, p_{t_q})\leq 4\cdot g^k_q\left(p^{(i)}_1, p^{(i)}_2, \ldots, p^{(i)}_{t_q}\right)$ for some $i$. For each $i\in [k]$, let $\alpha (i)$ be the number of steps in which we apply Lemma~\ref{le:reduce} for color $i$. By the definition of $g^k_q(p_1, p_2, \ldots, p_{t_q})$, we have $g^k_q(p_1, p_2, \ldots, p_{t_q})=1<2$ if $p_j=1$ for some $j\in [t_q]$. We also have $g^k_q(p_1, p_2, \ldots, p_{t_q})=2$ if $p_j=2$ for all $j\in [t_q]$ with $\left|I_j\right|=q$. Thus $\sum_{I\in \mathcal{I}, \left|I\right|=q}\sum_{i\in I} \alpha(i)\leq (p-2)\binom{k}{q}$. Then $\sum^{k}_{i=1} \alpha (i)=\frac{1}{\binom{k-1}{q-1}}\sum_{I\in \mathcal{I}, \left|I\right|=q}\sum_{i\in I} \alpha(i)\leq \frac{1}{\binom{k-1}{q-1}}\binom{k}{q}(p-2)=\frac{k(p-2)}{q}$. We conclude that $g^k_q(p)\leq4^{\frac{k(p-2)}{q}}\cdot 2= 2^{\frac{2k(p-2)}{q}+1}$, completing the proof of Theorem~\ref{th:q}.
\end{proof}

\begin{proof}[Proof of Proposition~\ref{prop:Kp}] The proof is similar to the proof of Theorem~\ref{th:q}. The only difference is that we repeatedly apply Lemma \ref{le:reduce} until in some step we get $g^k_1(p_1, p_2, \ldots, p_{t_q})<32$. Note that we have $g^k_1(2)=2\leq 32$, $g^k_1(2, \ldots, 2, 6)=6< 32$, $g^k_1(2, \ldots, 2, 3, 5)=R(K_3, K_5)=14< 32$ (\cite{GrGl}), $g^k_1(2, \ldots, 2, 4, 4)=R_2(K_4)=18<32$ (\cite{GrGl}), $g^k_1(2, \ldots, 2, 3, 3, 4)=g^3_1(3, 3, 4)=17< 32$ (\cite{LMSSS}) and $g^k_1(2, \ldots, 2, 3, 3, 3, 3)=g^4_1(3)=26<32$ (\cite{ChGr,GSSS}). Thus we have $\sum^{k}_{i=1} \alpha (i)\leq k(p-2)-4$ in this case, so $g^k_1(p)< 4^{k(p-2)-4}\cdot 32= 2^{2k(p-2)-3}$.
\end{proof}

In the rest of this section, we prove Proposition \ref{prop:sp}, using a similar method to that used in \cite{WMLC}.

\begin{proof}[Proof of Proposition~\ref{prop:sp}]
Consider a $k$-edge-coloring $G$ of $K_n$, where each edge receives color $i$ ($1\leq i\leq k-1$) with probability $\frac{r}{k-1}$ and
color $k$ with probability $1-r$ (for small $r$, to be determined shortly), and these probabilities are mutually independent. For each set $S$ of $s$ vertices, let $A_S$ be the event that $G[S]$ is a rainbow $K_s$. For each set $T$ of $p$ vertices, let $B_T$ be the event that $G[T]$ is a $q$-colored $K_p$. We shall show that ${\rm Pr}((\wedge_S \overline{A}_S)\wedge (\wedge_T \overline{B}_T))> 0$.

Define a graph $D$ with a vertex set corresponding to all possible $A_S$ and $B_T$ such that (the vertex corresponding to) $A_S$ is adjacent to (the vertex corresponding to) $B_T$ if and only if $\left|S\cap T\right|\geq 2$, and $A_S$ (resp., $B_T$) is adjacent to $A_{S'}$ (resp., $B_{T'}$) if and only if $\left|S\cap S'\right|\geq 2$ (resp., $\left|T\cap T'\right|\geq 2$). Then $D$ is a dependency graph. We define $N_{AA}$, $N_{AB}$, $N_{BA}$ and $N_{BB}$ such that $N_{XY}$ is the number of vertices in $D$ of type $Y$ (so corresponding either to a number of $A_S$ vertices or a number of $B_T$ vertices) adjacent to a fixed vertex of type $X$ (so either one $A_S$ vertex or one $B_T$ vertex). In order to be able to apply  Lemma~\ref{le:local}, for each $S$, let the positive real number $y_i=y$ correspond to event $A_S$, and for each $T$, let $y_i=z$ correspond to event $B_T$. By Lemma \ref{le:local}, to show that ${\rm Pr}((\wedge_S \overline{A}_S)\wedge (\wedge_T \overline{B}_T))> 0$, it suffices to show that there exist positive real numbers $r, y, z$ such that
\begin{equation}\label{eq:sp1}
  r<1,\ y{\rm Pr}(A_s)<1,\ z{\rm Pr}(B_T)<1,
\end{equation}
\begin{equation}\label{eq:sp2}
  \ln y> y{\rm Pr}(A_s)N_{AA}+ z{\rm Pr}(B_T)N_{AB},
\end{equation}
\begin{equation}\label{eq:sp3}
  \ln z> y{\rm Pr}(A_s)N_{BA}+ z{\rm Pr}(B_T)N_{BB}.
\end{equation}

Note that for $r$ small, we have
\begin{align*}
 {\rm Pr}(A_S)\leq &~\binom{k-1}{\binom{s}{2}}\binom{s}{2}!\left(\frac{r}{k-1}\right)^{\binom{s}{2}} +\binom{k-1}{\binom{s}{2}-1}\left(\binom{s}{2}-1\right)!\binom{s}{2}(1-r)\left(\frac{r}{k-1}\right)^{\binom{s}{2}-1}\\
   = &~\binom{s}{2}(k-1)(k-2)\cdots\left(k-\binom{s}{2}+1\right)\left(\frac{r}{k-1}\right)^{\binom{s}{2}-1} \left(\frac{k-\binom{s}{2}}{\binom{s}{2}}\cdot\frac{r}{k-1}+1-r\right)\\
  \leq &~Lr^{\binom{s}{2}-1}
\end{align*}
and
\begin{align*}
 {\rm Pr}(B_T)\leq &~\binom{k-1}{q}\left(\frac{qr}{k-1}\right)^{\binom{p}{2}} +\binom{k-1}{q-1}\left(1-r+\frac{(q-1)r}{k-1}\right)^{\binom{p}{2}}\\
   \leq &~\binom{k-1}{q}\left(\frac{r}{2}\right)^{\binom{p}{2}} +\binom{k-1}{q-1}\left(1-\frac{r}{2}\right)^{\binom{p}{2}}\\
   \leq &~\left(\binom{k-1}{q}+\binom{k-1}{q-1}\right)\left(1-\frac{r}{2}\right)^{\binom{p}{2}}\\
   \leq &~\binom{k}{q}\exp \left(-\frac{r}{2}\binom{p}{2}\right) =~\exp \left(-\frac{rp^2}{4}+\frac{rp}{4}+\ln \binom{k}{q}\right).
\end{align*}
We bound $N_{AA}$, $N_{AB}$, $N_{BA}$ and $N_{BB}$ as follows:
\begin{equation*}\label{eq:sp4}
  N_{AA}\leq \binom{s}{2}\binom{n-2}{s-2}\leq s^2n^{s-2},\ N_{AB}\leq \binom{s}{2}\binom{n-2}{p-2}\leq s^2n^{p-2},
\end{equation*}
\begin{equation*}\label{eq:sp5}
  N_{BA}\leq \binom{p}{2}\binom{n-2}{s-2}\leq p^2n^{s-2},\ N_{BB}\leq \binom{p}{2}\binom{n-2}{p-2}\leq p^2n^{p-2}.
\end{equation*}
Let $\alpha=(s-2)/\left(\binom{s}{2}-2.1\right)$ and $\beta=1/\left(\binom{s}{2}-1\right)$. We set
\begin{equation*}\label{eq:sp6}
  r=c_1n^{-\alpha}L^{-\beta},\ p=c_2n^{\alpha}(\ln n)L^{\beta},\ y=1+\epsilon,\ z=\exp \left(c_3n^{\alpha}(\ln n)^2L^{\beta}\right),
\end{equation*}
where $\epsilon \ll 1$, $c_1, c_2, c_3$ are appropriately chosen and $n$ tends to infinity. Then we have
\begin{align*}
 y{\rm Pr}(A_S)N_{AA}\leq &~(1+\epsilon)Lr^{\binom{s}{2}-1}s^2n^{s-2}=~(1+\epsilon)s^2c_1^{\binom{s}{2}-1}n^{\frac{-1.1(s-2)}{\binom{s}{2}-2.1}},
\end{align*}
\begin{align*}
 y{\rm Pr}(A_S)N_{BA}\leq &~(1+\epsilon)Lr^{\binom{s}{2}-1}p^2n^{s-2}=~(1+\epsilon)c_1^{\binom{s}{2}-1}c_2^2L^{2\beta}n^{\alpha-\frac{0.1(s-2)}{\binom{s}{2}-2.1}}(\ln n)^2,
\end{align*}
\begin{align*}
 z{\rm Pr}(B_T)N_{AB}\leq &~\exp \left(c_3n^{\alpha}(\ln n)^2L^{\beta}-\frac{rp^2}{4}+\frac{rp}{4}+\ln \binom{k}{q}+2\ln s+(p-2) \ln n\right)\\
 \leq &~\exp \left(c_3n^{\alpha}(\ln n)^2L^{\beta}-\frac{c_1c_2^2}{4}n^{\alpha}(\ln n)^2L^{\beta}+o(n^{\alpha}(\ln n)^2)+c_2n^{\alpha}(\ln n)^2L^{\beta}\right)\\
 \leq &~\exp \left(\left(c_3-\frac{c_1c_2^2}{4}+c_2+o(1)\right)n^{\alpha}(\ln n)^2L^{\beta}\right),
\end{align*}
and
\begin{align*}
 z{\rm Pr}(B_T)N_{BB}\leq &~\exp \left(c_3n^{\alpha}(\ln n)^2L^{\beta}-\frac{rp^2}{4}+\frac{rp}{4}+\ln \binom{k}{q}+2\ln p+(p-2) \ln n\right)\\
 \leq &~\exp \left(\left(c_3-\frac{c_1c_2^2}{4}+c_2+o(1)\right)n^{\alpha}(\ln n)^2L^{\beta}\right).
\end{align*}
If we choose $c_1, c_2, c_3$ such that $c_3-\frac{c_1c_2^2}{4}+c_2+o(1)<0$, then inequations (\ref{eq:sp1})-(\ref{eq:sp3}) hold. Setting $c=c_2$ in the above expression for $p$, and expressing $n$ in terms of $p$, we have
$$n\geq\left(\frac{(s-2)pL^{1/\left(1-\binom{s}{2}\right)}}{(c+o(1))\left(\binom{s}{2}-2.1\right) \ln \left(pL^{1/\left(1-\binom{s}{2}\right)}\right)}\right)^{\left(\binom{s}{2}-2.1\right)/(s-2)}.$$
\end{proof}

\section{Proofs of Theorems \ref{th:q=p-2} and \ref{th:q=sqrt}}
\label{sec:ch-proof q=p-2}

We first present our proof of Theorem~\ref{th:q=p-2}.

\begin{proof}[Proof of Theorem~\ref{th:q=p-2}] We first show that there is a Gallai-$k$-coloring of $K_{k+1}$, in which there is no $K_p$ receiving at most $p-2$ distinct colors. The case $k=p-2$ is trivial since $K_{p-1}$ contains no $K_p$. For $k\geq p-1$, let $V(K_{k+1})=\{v_1, v_2, \ldots, v_{k+1}\}$. For every $1\leq i< j\leq k+1$, we color the edge $v_iv_j$ using color $i$. Note that for any three vertices $v_i, v_j, v_k$ with $i< j< k$, we have $c(v_iv_j)=c(v_iv_k)$, so there are no rainbow triangles. For any $p$ vertices $v_{i_1}, v_{i_2}, \ldots, v_{i_p}$ with $i_1< i_2<\cdots <i_p$, we have $C(\{v_{i_1}, v_{i_2}, \ldots, v_{i_p}\})=\{i_1, i_2, \ldots, i_{p-1}\}$, so every $K_p$ receives $p-1$ distinct colors.

Next, we show that $g^k_{p-2}(p)\leq k+2$ by induction on $k$. For the base case, if $k=p-2$, then it is trivial that $g^k_{p-2}(p)=p$. Now assume that it holds for every $p-2\leq k'\leq k-1$, and we will prove it for $k$.

For a contradiction, suppose that $G$ is a Gallai-$k$-coloring of $K_{k+2}$ without a $(p-2)$-colored $K_p$. Using Theorem \ref{th:Gallai}, let $V_1, V_2, \ldots, V_m$ ($m\geq 2$) be a Gallai-partition of $V(G)$. Note that $m\leq p-1$ since $p-2\geq 2$. If $m\geq 4$, then we can choose nonempty subsets $V'_i\subseteq V_i$ ($1 \leq i\leq m$) such that $\sum^{m}_{i=1}\left|V'_i\right|=p$. Since $G$ is a Gallai-coloring, we have $\left|C(V'_i)\right|\leq \left|V'_i\right|-1$ ($1 \leq i\leq m$) by Theorem \ref{th:antiR}. Then $\left|C(\bigcup^{m}_{i=1} V'_i)\right|\leq 2+\sum^{m}_{i=1}(\left|V'_i\right|-1)=2+p-m\leq p-2$. Thus there is a $(p-2)$-colored $K_p$ in $G$, a contradiction. Hence, we have $m\leq 3$. Note that if $G$ contains a Gallai-partition with exactly three parts, then $G$ also contains a Gallai-partition with exactly two parts. Thus we may assume that $m=2$ and $c(V_1, V_2)=1$.

\noindent\begin{claim}\label{cl:q=p-2 1} $1\notin C(V_1)$ and $1\notin C(V_2)$.
\end{claim}

\begin{proof} By symmetry, we only prove $1\notin C(V_1)$. If $1\in C(V_1)$, then we may choose $V'_1\subseteq V_1$ and $V'_2\subseteq V_2$ such that $1\in C(V'_1)$ and $\left|V'_1\right|+\left|V'_2\right|=p$. Thus $\left|C(V'_1\cup V'_2)\right|\leq \left|C(V'_1)\right|+\left|C(V'_2)\right|\leq \left|V'_1\right|-1+\left|V'_2\right|-1=p-2$, a contradiction.
\end{proof}

\noindent\begin{claim}\label{cl:q=p-2 2} $\left|V_1\right|= \left|C(V_1)\right| +1$ and $\left|V_2\right|= \left|C(V_2)\right| +1$.
\end{claim}

\begin{proof} By symmetry, we only prove it for $V_1$. By Theorem \ref{th:antiR}, we have $\left|V_1\right|\geq \left|C(V_1)\right| +1$, so it suffices to prove $\left|V_1\right|\leq \left|C(V_1)\right| +1$. Suppose for a contradiction that $\left|V_1\right|\geq \left|C(V_1)\right| +2$. If $\left|C(V_1)\right|\leq p-3$, then $\left|V_1\right|\leq p-1$ in order to avoid a $(p-2)$-colored $K_p$. Thus we can choose $V'_2\subseteq V_2$ with $\left|V_1\right|+\left|V'_2\right|=p$. Since $\left|C(V'_2)\right|\leq \left|V'_2\right|-1$, we have $\left|C(V_1\cup V'_2)\right|\leq 1+\left|C(V_1)\right|+\left|V'_2\right|-1\leq \left|C(V_1)\right|+p-\left|V_1\right|\leq \left|C(V_1)\right|+p-(\left|C(V_1)\right| +2)=p-2$, a contradiction. Thus $\left|C(V_1)\right|\geq p-2$, and then we have $\left|V_1\right|\leq \left|C(V_1)\right| +1$ by Claim \ref{cl:q=p-2 1} and the induction hypothesis.
\end{proof}

We now show that $C(V_1)\cap C(V_2)=\emptyset$. Otherwise, suppose $2\in C(V_1)\cap C(V_2)$. We choose $V'_1\subseteq V_1$ and $V'_2\subseteq V_2$ such that $2\in C(V'_1)$, $2\in C(V'_2)$ and $\left|V'_1\right|+\left|V'_2\right|=p$. Then $\left|C(V'_1\cup V'_2)\right|\leq 1+\left|C(V'_1)\right|+\left|C(V'_2)\right|-1 \leq \left|V'_1\right|-1+\left|V'_2\right|-1=p-2$, a contradiction. Finally, by Claims \ref{cl:q=p-2 1} and \ref{cl:q=p-2 2}, we have $k+2=\left|V(G)\right|=\left|V_1\right|+\left|V_2\right|=\left|C(V_1)\right|+1+\left|C(V_2)\right|+1\leq k-1+2=k+1$, a contradiction.
\end{proof}

In the following, instead of proving Theorem~\ref{th:q=sqrt}, we will prove the following more general result.

\noindent\begin{theorem}\label{th:q=sqrt+}
For integers $p\gg m\geq 2$ and $k\geq \left\lfloor \sqrt[m]{p-1}\right\rfloor-1$, we have $g^k_{\left\lfloor \sqrt[m]{p-1}\right\rfloor-1}(p)\geq (k+1)^m+1$.
\end{theorem}

\begin{proof} Let $q= \left\lfloor \sqrt[m]{p-1}\right\rfloor-1$. By Theorem \ref{th:q=p-2}, we have $g^k_{q}(q+2)>k+1$. Let $G_0$ be a Gallai-$k$-coloring of $K_{k+1}$ in which the largest $q$-colored complete subgraph has order at most $q+1$, and let $G_1=G_0$. Suppose that for some $1\leq i<m$ we have constructed a $k$-edge-coloring $G_i$ of $K_{(k+1)^i}$. Then we construct $G_{i+1}$ by substituting $k+1$ copies of $G_i$ into vertices of $G_0$. Finally, we obtain a $k$-edge-coloring $G_m$ of $K_{(k+1)^m}$. It is easy to check that $G_m$ is a Gallai-coloring and that the largest $q$-colored complete subgraph has order at most $(q+1)^m \leq p-1$. Thus we have $g^k_q(p)\geq (k+1)^m+1$.
\end{proof}

\section{Proof of Theorem \ref{th:q=p-3}}
\label{sec:ch-proof q=p-3}

For the lower bound, we will construct a Gallai-$k$-coloring of $K_{k+2}$ without a $(p-3)$-colored $K_p$. The case $k=p-3$ is trivial since $K_{p-1}$ contains no $K_p$. For $k\geq p-2$, let $V(K_{k+2})=\{v_1, v_2, \ldots, v_{k+2}\}$. For every $1\leq i\leq k$ and $i< j\leq k+2$, we color the edge $v_iv_j$ using color $i$, and we color the edge $v_{k+1}v_{k+2}$ with color $k$. Then we obtain a desired edge-coloring.

For the upper bound, we will use induction on $k$. For the base case, if $k=p-3$, then it is trivial that $g^{k}_{p-3}(p)\leq k+3$. Now assume that it holds for every $p-3\leq k'\leq k-1$, and we will prove it for $k$. For a contradiction, suppose that $G$ is a Gallai-$k$-coloring of $K_{k+3}$ without a $(p-3)$-colored $K_p$. By the induction hypothesis, we may assume that all the $k$ colors appear in $G$ (that is, $C(G)=[k]$). Using Theorem \ref{th:Gallai}, let $V_1, V_2, \ldots, V_m$ ($m\geq 2$) be a Gallai-partition of $V(G)$. We choose it such that $m$ is minimum.

\medskip\noindent
{\bf Case 1.} $m\geq 4$.
\vspace{0.05cm}

\noindent
In this case, by the minimality of $m$, there are exactly two colors used between the parts, say colors 1 and 2. If $m\geq 5$, then we can choose one vertex $v_i$ from each $V_i$ ($1\leq i\leq 5$) to form a 2-colored $K_5$. Then we choose another $p-5$ vertices $v_{6}, v_{7}, \ldots, v_{p}$ one by one arbitrarily. Note that for each $6\leq i\leq p$, when we add vertex $v_i$ to $G_{i-1}=G[\{v_1, v_2, \ldots, v_{i-1}\}]$, we add at most one new color that is not used in $G_{i-1}$, by Lemma \ref{le:Gallai}. Thus we obtain a $(p-3)$-colored $K_p$, a contradiction. Hence, we have $m=4$.

\noindent\begin{claim}\label{cl:q=p-3 1} For any $i\in [4]$, we have $1, 2\notin C(V_i)$. For any $1\leq i<j\leq 4$, we have $C(V_i)\cap C(V_j)=\emptyset$.
\end{claim}

\begin{proof} If $C(V_i)\cap \{1,2\}\neq \emptyset$ for some $i\in [4]$, then we can choose nonempty subsets $V'_l\subseteq V_l$ ($1 \leq l\leq 4$) such that $\sum^{4}_{l=1}\left|V'_l\right|=p$ and $C(V'_i)\cap \{1,2\}\neq \emptyset$. Since $G$ is a Gallai-coloring, we have $\left|C(V'_l)\right|\leq \left|V'_l\right|-1$ ($1 \leq l\leq 4$) by Theorem \ref{th:antiR}. Then $|C(\bigcup^{4}_{l=1} V'_l)|\leq 2+(\sum^{4}_{l=1}\left|C(V'_l)\right|)-1\leq 2+(\sum^{4}_{l=1}(\left|V'_l\right|-1))-1=2+p-4-1= p-3$. Thus there is a $(p-3)$-colored $K_p$ in $G$, a contradiction. If $C(V_i)\cap C(V_j)\neq \emptyset$ for some $1\leq i<j\leq 4$, say $c_0\in C(V_i)\cap C(V_j)$, then we can choose nonempty subsets $V'_l\subseteq V_l$ ($1 \leq l\leq 4$) such that $\sum^{4}_{l=1}\left|V'_l\right|=p$, $c_0\in C(V'_i)$ and $c_0\in C(V'_j)$. Then $|C(\bigcup^{4}_{l=1} V'_l)|\leq 2+(\sum^{4}_{l=1}\left|C(V'_l)\right|)-1\leq 2+ (\sum^{4}_{l=1} (\left|V'_l\right|-1))-1=p-3$. Thus there is a $(p-3)$-colored $K_p$ in $G$, a contradiction.
\end{proof}

\noindent\begin{claim}\label{cl:q=p-3 2} For any $i\in [4]$, we have $\left|V_i\right|\leq \left|C(V_i)\right|+1$.
\end{claim}

\begin{proof} Suppose for a contradiction that $\left|V_i\right|\geq \left|C(V_i)\right|+2$ for some $i\in [4]$, say $i=1$. If $\left|C(V_1)\right|\leq p-5$, then $\left|V_1\right|\leq p-4$ in order to avoid a $(p-3)$-colored $K_p$. Thus we can choose nonempty subsets $V'_j\subseteq V_j$ ($2\leq j\leq 4$) such that $\left|V_1\right|+\sum^{4}_{j=2}|V'_j|=p$. Then $|C(V_1\cup (\bigcup^{4}_{j=2}V'_j))|\leq 2+\left|C(V_1)\right|+\sum^{4}_{j=2}(|V'_j|-1)\leq 2+\left|C(V_1)\right|+(p-\left|V_1\right|)-3\leq 2+\left|C(V_1)\right|+p-(\left|C(V_1)\right| +2)-3=p-3$, a contradiction.

If $\left|C(V_1)\right|\geq p-3$, then $\left|V_1\right|\leq \left|C(V_1)\right|+2$ by Claim \ref{cl:q=p-3 1} and the induction hypothesis. If $\left|C(V_1)\right|= p-4$, then $\left|V_1\right|\leq p-2=\left|C(V_1)\right|+2$ in order to avoid a $(p-3)$-colored $K_p$. Thus $\left|V_1\right|= \left|C(V_1)\right|+2$ whenever $\left|C(V_1)\right|\geq p-4$. By Theorem \ref{th:q=p-2}, there is a $(p_1-2)$-colored $K_{p_1}$ in $G[V_1]$ for every $4\leq p_1\leq \left|V_1\right|$. Let $H$ be a copy of a $(p-5)$-colored $K_{p-3}$ in $G[V_1]$. Then we can choose one vertex from each $V_j$ ($2\leq j\leq 4$) such that these vertices together with $H$ form a $(p-3)$-colored $K_p$, a contradiction.
\end{proof}

By Claims \ref{cl:q=p-3 1} and \ref{cl:q=p-3 2}, we have $k+3= \left|V(G)\right|=\sum^{4}_{i=1}\left|V_i\right|\leq \sum^{4}_{i=1}(\left|C(V_i)\right|+1)\leq k-2+4=k+2$, a contradiction.

\medskip\noindent
{\bf Case 2.} $2\leq m\leq 3$.
\vspace{0.05cm}

\noindent
By the minimality of $m$, we may assume that $m=2$ and $c(V_1, V_2)=1$.

\noindent\begin{claim}\label{cl:q=p-3 3'} At most one of $V_1$ and $V_2$ contains an edge with color 1.
\end{claim}

\begin{proof} If $1\in C(V_1)$ and $1\in C(V_2)$, then we can choose $V'_1\subseteq V_1$ and $V'_2\subseteq V_2$ such that $\left|V'_1\right|+\left|V'_2\right|=p$, $1\in C(V'_1)$ and $1\in C(V'_2)$. Then $\left|C(V'_1\cup V'_2)\right|\leq \left|C(V'_1)\right|+\left|C(V'_2)\right|-1\leq \left|V'_1\right|-1+\left|V'_2\right|-1-1=p-3$, a contradiction.
\end{proof}

\noindent\begin{claim}\label{cl:q=p-3 3} We have $\left|V_i\right|=\left|C(V_i)\right|+1$ and $\left|V_{3-i}\right|=\left|C(V_{3-i})\right|+2$ for some $i\in [2]$.
\end{claim}

\begin{proof} Recall that $\left|V_i\right|\geq \left|C(V_i)\right|+1$ for each $i\in [2]$ by Theorem \ref{th:antiR}. First suppose that $\left|V_i\right|\geq \left|C(V_i)\right|+2$ for all $i\in [2]$. Note that for each $i\in [2]$, since $\left|V_i\right|\geq 2$, we have $\left|C(V_i)\right|\geq 1$ and thus $\left|V_i\right|\geq 3$. Moreover, if $\left|C(V_1)\right|= 1$ (resp., $\left|C(V_2)\right|= 1$), then $G[V_1]$ (resp., $G[V_2]$) is a monochromatic complete subgraph of order at least $3$, and if $\left|C(V_1)\right|\geq 2$ (resp., $\left|C(V_2)\right|\geq 2$), then $G[V_1]$ (resp., $G[V_2]$) contains a $(p'-2)$-colored $K_{p'}$ for every $4\leq p'\leq \left|V_1\right|$ (resp., $4\leq p'\leq \left|V_2\right|$) by Theorem \ref{th:q=p-2}. Thus we can choose a $(p_i-2)$-colored $K_{p_i}$ in $G[V_i]$ for each $i\in [2]$ such that $3\leq p_i\leq \left|V_i\right|$ and $p_1+p_2=p$, so there is a $(p-3)$-colored $K_p$ in $G$, a contradiction. Hence, we may assume that $\left|V_1\right|=\left|C(V_1)\right|+1$ without loss of generality.

If $\left|V_2\right|= \left|C(V_2)\right|+1$, then $k+3=\left|V_1\right|+\left|V_2\right|=\left|C(V_1)\right|+\left|C(V_2)\right|+2$, so $\left|C(V_1)\right|+\left|C(V_2)\right|=k+1$. Then $C(V_1)\cap C(V_2)\neq \emptyset$. Let $C'= C(V_1)\cap C(V_2)$, and we have $1\notin C'$ by Claim \ref{cl:q=p-3 3'}. If $1\notin C(V_1)$ and $1\notin C(V_2)$, then $\left|C'\right|\geq 2$ (otherwise we have $\left|C(V_1)\right|+\left|C(V_2)\right|\leq k$). Then we can choose $V'_1\subseteq V_1$ and $V'_{2}\subseteq V_{2}$ with $\left|V'_1\cup V'_2\right|=p$ such that $\left|C(V'_1)\cap C(V'_2)\right|\geq 2$. Then $\left|C(V'_1\cup V'_2)\right|\leq 1+\left|C(V'_1)\right|+\left|C(V'_2)\right|-2\leq 1+\left|V'_1\right|-1+\left|V'_2\right|-1-2=p-3$, a contradiction. Hence, without loss of generality, we may assume that $1\in C(V_1)$, $c_0\in C'$ and $c_0\neq 1$. Then we can choose $V'_1\subseteq V_1$ and $V'_{2}\subseteq V_{2}$ with $\left|V'_1\cup V'_2\right|=p$ such that $\{1, c_0\}\subseteq C(V'_1)$ and $c_0\in C(V'_{2})$. Then $\left|C(V'_1\cup V'_2)\right|\leq \left|C(V'_1)\right|+\left|C(V'_2)\right|-1\leq \left|V'_1\right|-1+\left|V'_2\right|-1-1=p-3$, a contradiction. Therefore, we have $\left|V_2\right|\geq \left|C(V_2)\right|+2$.

If $1\leq \left|C(V_2)\right|\leq p-4$, then $\left|V_2\right|\leq p-2$ in order to avoid a $(p-3)$-colored $K_p$. Let $V'_1\subseteq V_1$ such that $\left|V'_1\cup V_2\right|=p$. Then $\left|V_2\right|=p-\left|V'_1\right|\leq p-(\left|C(V'_1)\right|+1)\leq p-(\left|C(V'_1\cup V_2)\right|-\left|C(V_2)\right|)\leq p-(p-2-\left|C(V_2)\right|)=\left|C(V_2)\right|+2$, where the second inequality is by $\left|C(V'_1\cup V_2)\right|\leq 1+\left|C(V'_1)\right|+\left|C(V_2)\right|$, and the last inequality follows from the assumption that $G$ contains no $(p-3)$-colored $K_p$. If $p-3\leq \left|C(V_2)\right|\leq k-1$, then by the induction hypothesis we have $\left|V_2\right|\leq \left|C(V_2)\right|+2$. If $\left|C(V_2)\right|=k$, then $\left|V_2\right|=k+3-\left|V_1\right|\leq k+2=\left|C(V_2)\right|+2$. Therefore, we have $\left|V_2\right|=\left|C(V_2)\right|+2$.
\end{proof}

By Claim \ref{cl:q=p-3 3}, we may assume that $\left|V_1\right|= \left|C(V_1)\right|+1$ and $\left|V_2\right|= \left|C(V_2)\right|+2$ without loss of generality.

\noindent\begin{claim}\label{cl:q=p-3 4} $C(V_1)\cap C(V_2)=\emptyset$.
\end{claim}

\begin{proof} For a contradiction, suppose that $C'= C(V_1)\cap C(V_2)\neq \emptyset$. Similar to the second paragraph in the proof of Claim \ref{cl:q=p-3 3}, we have $1\notin C(V_1)$, $1\notin C(V_2)$ and $\left|C'\right|=1$, say $C'=\{c_0\}$.

If $\left|V_2\right|\leq p-2$, then we can choose $V'_1\subseteq V_1$ such that $\left|V'_1\cup V_2\right|=p$ and $c_0\in C(V'_1)$. Now we have $\left|C(V'_1\cup V_2)\right|\leq 1+\left|C(V'_1)\right|+\left|C(V_2)\right|-1\leq 1+(\left|V'_1\right|-1)+(\left|V_2\right|-2)-1=p-3$, a contradiction. Thus $\left|V_2\right|\geq p-1$ and $\left|C(V_2)\right|=\left|V_2\right|-2\geq p-3$. Let $uv$ be an edge within $V_1$ with $c(uv)=c_0$. Then $G[V_2\cup\{u,v\}]$ is a $(\left|C(V_2)\right|+1)$-colored $K_{\left|V_2\right|+2}$. If $\left|C(V_1)\right|\geq 2$, then $\left|C(V_2\cup\{u,v\})\right|\leq k-1$, and thus we can derive a contradiction by the induction hypothesis. Thus we have $C(V_1)=\{c_0\}$ and $\left|V_1\right|=2$.

By Theorem \ref{th:q=p-2}, we may assume that $H$ is a copy of a $(p-5)$-colored $K_{p-3}$ in $G[V_2]$. If $c_0\in C(H)$, then $G[V(H)\cup V_1]$ is a $(p-4)$-colored $K_{p-1}$. For any vertex $w\in V_2\setminus V(H)$, we have $\left|C(w, V(H)\cup V_1)\setminus C(V(H)\cup V_1)\right|\leq 1$ by Lemma \ref{le:Gallai}, which implies a $(p-3)$-colored $K_{p}$, a contradiction. If $c_0\notin C(H)$ and there is an edge $xy$ with color $c_0$ such that $x\in V(H)$ and $y\in V_2\setminus V(H)$, then $C(y, V(H))\setminus C(H)=\{c_0\}$ by Lemma \ref{le:Gallai}. Then $G[V(H)\cup V_1\cup \{y\}]$ is a $(p-3)$-colored $K_{p}$, a contradiction. Hence, $G[V_2]$ contains no edge in color 1 which has an end-vertex in $V(H)$. Thus we may assume that $xy$ is an edge with color $c_0$ such that $x, y\in V_2\setminus V(H)$. By Theorem \ref{th:q=p-2}, we may further assume that $H'$ is a copy of a $(p-6)$-colored $K_{p-4}$ in $H$. By Lemma \ref{le:Gallai}, we have $\left|C(x, V(H'))\setminus C(H')\right|\leq 1$ and $C(y, V(H')\cup\{x\})\setminus C(V(H')\cup\{x\})=\{c_0\}$. Then $G[V(H')\cup V_1\cup \{x, y\}]$ is a $(p-3)$-colored $K_{p}$, a contradiction.
\end{proof}

By Claim \ref{cl:q=p-3 3}, we have $\left|C(V_1)\right|+\left|C(V_2)\right|=\left|V_1\right|+\left|V_2\right|-3=k$. Then we have either $1\in C(V_1)$ or $1\in C(V_2)$ by Claims \ref{cl:q=p-3 3'} and \ref{cl:q=p-3 4}. We first consider the case $1\in C(V_1)$ and $1\notin C(V_2)$. We define a subset $V'_2\subseteq V_2$ as follows. If $\left|V_2\right|\leq p-3$, then $V'_2=V_2$. If $\left|V_2\right|\geq p-2$, then we choose $V'_2$ such that $G[V'_2]$ is a $(p-4)$-colored $K_{p-2}$ (using Theorem \ref{th:q=p-2}). Then let $V'_1\subset V_1$ such that $\left|V'_1\right|=p-\left|V'_2\right|$ and $1\in C(V'_1)$. Since $\left|C(V'_1\cup V'_2)\right|\leq \left|C(V'_1)\right|+\left|C(V'_2)\right|\leq \left|V'_1\right|-1+\left|V'_2\right|-2=p-3$, we derive a contradiction. Next, we consider the case $1\notin C(V_1)$ and $1\in C(V_2)$. In this case, we have $\left|V_2\right|\geq p$, since otherwise if $\left|V_2\right|\leq p-1$, then we can choose $V'_1\subset V_1$ with $\left|V'_1\cup V_2\right|=p$ such that $\left|C(V'_1\cup V_2)\right|\leq \left|C(V'_1)\right|+\left|C(V_2)\right|\leq \left|V'_1\right|-1+\left|V_2\right|-2=p-3$, a contradiction.

By Theorem \ref{th:q=p-2}, we may assume that $H$ is a copy of a $(p-4)$-colored $K_{p-2}$ in $G[V_2]$. Let $u$ be any vertex in $V_1$. If $1\in C(H)$, then $G[V(H)\cup\{u\}]$ is a $(p-4)$-colored $K_{p-1}$. For any vertex $v\in V_2\setminus V(H)$, we have $\left|C(v, V(H)\cup\{u\})\setminus C(V(H)\cup\{u\})\right|\leq 1$ by Lemma \ref{le:Gallai}, which implies a $(p-3)$-colored $K_{p}$, a contradiction. If $1\notin C(H)$ and there is an edge $xy$ with color $1$ such that $x\in V(H)$ and $y\in V_2\setminus V(H)$, then $C(y, V(H))\setminus C(H)=\{1\}$ by Lemma \ref{le:Gallai}. Then $G[V(H)\cup \{u, y\}]$ is a $(p-3)$-colored $K_{p}$, a contradiction. If $1\notin C(H)$ and $G[V_2]$ contains no edge with color $1$ incident with a vertex of $H$, then we may assume that $xy$ is an edge with color $1$ such that $x, y\in V_2\setminus V(H)$. By Theorem \ref{th:q=p-2}, we may further assume that $H'$ is a copy of a $(p-5)$-colored $K_{p-3}$ in $H$. By Lemma \ref{le:Gallai}, we have $\left|C(x, V(H'))\setminus C(H')\right|\leq 1$ and $C(y, V(H')\cup\{x\})\setminus C(V(H')\cup\{x\})=\{1\}$. Then $G[V(H')\cup \{x, y, u\}]$ is a $(p-3)$-colored $K_{p}$. This contradiction completes the proof of Theorem \ref{th:q=p-3}. \qed

\noindent\begin{remark}\label{re:q=p-3}
The bound $p\geq 8$ in Theorem~\ref{th:q=p-3} is best possible. Indeed, if $p=7$, then we can show that $g^5_4(7)>8$ by the following counterexample. Let $G_1$ (resp., $G_2$) be a $K_4$ using colors 1 and 2 (resp., colors 3 and 4) such that colors 1 and 2 (resp., colors 3 and 4) induce two monochromatic copies of a $P_4$. Let $G$ be a 5-colored $K_8$ obtained by joining $G_1$ and $G_2$ using edges that all get color 5. It is easy to check that $G$ contains neither a rainbow $K_3$ nor a 4-colored $K_7$. For the case that $p=6$, we can prove that $g^4_3(6)=8$ and $g^5_3(6)=10$ (the proofs are given in the appendices). When $p=5$, the function $g^k_2(5)$ is exponential in $k$ by Theorem~\ref{th:p=5}.
\end{remark}

\section{Proofs of Theorems \ref{th:p=5} and \ref{th:q=log}}
\label{sec:ch-proof p=5}

We first present our proof of Theorem \ref{th:p=5}.

\begin{proof}[Proof of Theorem \ref{th:p=5}] We first show that there is a Gallai-$k$-coloring of $K_{2^k}$, in which there is no $K_5$ receiving
at most two distinct colors. For $k=2$, let $G_2$ be an edge-coloring of $K_4$ with colors 1 and 2 such that color 1 induces a perfect matching and color 2 induces a $C_4$. It is easy to check that there is neither a rainbow $K_3$ nor a monochromatic $K_3$ in $G_2$, and $G_2$ contains no 2-colored $K_5$ clearly. Suppose for some $2\leq i\leq k-1$ we have constructed a Gallai-$i$-coloring $G_i$ of $K_{2^i}$ in which there is neither a monochromatic $K_3$ nor a 2-colored $K_5$. Then we construct an $(i+1)$-edge-coloring $G_{i+1}$ of $K_{2^{i+1}}$ by joining two copies of $G_{i}$ with edges that all get color $i+1$. Since $G_{i}$ contains no rainbow $K_3$, there is no rainbow $K_3$ in $G_{i+1}$. Since $G_{i}$ contains neither a monochromatic $K_3$ nor a 2-colored $K_5$, there is no 2-colored $K_5$ in $G_{i+1}$. By repeating this process, we finally obtain a Gallai-$k$-coloring $G_k$ of $K_{2^k}$ without a 2-colored $K_5$.

We now prove that $g^k_2(5)\leq2^k+1$ by induction on $k$. For the base case, it is trivial that $g^2_2(5)=5$. Now assume that it holds for every $2\leq k'\leq k-1$, and we will prove it for $k\geq 3$.

For a contradiction, suppose that $G$ is a Gallai-$k$-coloring of $K_{2^k+1}$ without a 2-colored $K_5$. Using Theorem \ref{th:Gallai}, let $V_1, V_2, \ldots, V_m$ ($m\geq 2$) be a Gallai-partition of $V(G)$. We choose it such that $m$ is minimum. Since there is no 2-colored $K_5$, we have $m\leq 4$. If $m=4$, then by the minimality of $m$, there are exactly two colors used between the parts, say colors 1 and 2. In order to avoid a 2-colored $K_5$, there is no edge with color 1 or 2 within each part. If $k=3$, then there is only color 3 within these parts. Note that $\max_{1\leq i< j\leq 4}\left|V_i\cup V_j\right|\geq 5$, so there is a 2-colored $K_5$. Thus $k\geq 4$. By the induction hypothesis, we have $\left|V(G)\right|=\left|V_1\right|+\left|V_2\right|+\left|V_3\right|+\left|V_4\right|\leq 4\cdot 2^{k-2}=2^k$, a contradiction.

Thus $2\leq m\leq 3$, and by the minimality of $m$ we may assume $m=2$ and $c(V_1, V_2)=1$. If $1\notin C(V_1)$ and $1\notin C(V_2)$, then by the induction hypothesis, we have $\left|V(G)\right|=\left|V_1\right|+\left|V_2\right|\leq 2^{k-1}+2^{k-1}=2^k$, a contradiction. If $1\in C(V_1)$ and $1\in C(V_2)$, then $G$ contains a monochromatic $K_4$. By Lemma \ref{le:Gallai}, $G$ contains a $2$-colored $K_5$, a contradiction. Thus we may assume that $1\in C(V_1)$ and $1\notin C(V_2)$ without loss of generality.

\noindent\begin{claim}\label{cl:p=5 2} Color 1 induces a bipartite graph within $V_1$.
\end{claim}

\begin{proof} We first show that $G[V_1]$ contains no monochromatic $K_3$ with color 1. Otherwise, suppose $\{u, v, w\}$ forms a monochromatic $K_3$ with color 1 within $V_1$. Then for any vertex $x\in V_2$, we have that $\{u, v, w, x\}$ forms a monochromatic $K_4$. By Lemma \ref{le:Gallai}, there is a 2-colored $K_5$ in $G$, a contradiction.

We next show that $G[V_1]$ contains no $C_4$ with exactly three edges in color 1. Otherwise, if $G[V_1]$ contains such $C_4$, say $c(uv)=c(vw)=c(wz)=1$ and $c(zu)=2$. In order to avoid a rainbow $K_3$, we have $c(uw)\in \{1, 2\}$ and $c(vz)\in \{1, 2\}$. Then for any vertex $x\in V_2$, we have that $\{u, v, w, z, x\}$ forms a 2-colored $K_5$, a contradiction.

Finally, we show that $G[V_1]$ contains no monochromatic odd cycle in color 1 (thus color 1 induces a bipartite graph within $V_1$). Suppose that $C_{2t+1}=a_1a_2\cdots a_{2t+1}a_1$ ($t\geq 2$) is a monochromatic cycle using color 1 in $G[V_1]$. Since there is no $C_4$ with exactly three edges in color 1, we have $c(a_1a_4)=1$, so $c(a_1a_6)=1$, $c(a_1a_8)=1, \cdots, c(a_1a_{2t})=1$. Then $\{a_1, a_{2t}, a_{2t+1}\}$ forms a monochromatic $K_3$ in color 1, a contradiction.
\end{proof}

Let $E_1$ be the set of edges with color 1 in $G[V_1]$, and let $V'_1\subseteq V_1$ be the set of vertices incident with some edge of $E_1$. By Claim \ref{cl:p=5 2}, we may partition $V'_1$ into two parts $A$ and $B$ such that $1\notin C(A)$ and $1\notin C(B)$. Since $1\in C(V_1)$, we have $A\neq \emptyset$ and $B\neq \emptyset$. Let $V''_1=V_1\setminus V'_1$ (it is possible that $V''_1=\emptyset$).

\noindent\begin{claim}\label{cl:p=5 3} The following statements hold:
\begin{itemize}
  \item[{\rm (1)}] for any color $i\in C(V'_1)$, we have $i\notin C(V_2)$;
  \item[{\rm (2)}] $\left|V_2\right|\leq 2^{\left|C(V_2)\right|}$;
  \item[{\rm (3)}] $\left|A\right|\leq 2^{\left|C(A)\right|}$ and $\left|B\right|\leq 2^{\left|C(B)\right|}$.
\end{itemize}
\end{claim}

\begin{proof} (1) If $i=1$, then it holds clearly. If $i\neq 1$, then we may assume that $c(uv)=i$ for some $u, v\in V'_1$. Since $u\in V'_1$, there exists a vertex $w\in V'_1\setminus \{u,v\}$ with $c(uw)=1$. In order to avoid a rainbow $K_3$, we have $c(vw)\in \{1, i\}$. If $i\in C(V_2)$, then there is a 2-colored $K_5$ using colors 1 and $i$ in $G$, a contradiction.

(2) Let $\left|C(V_2)\right|=j$ ($0\leq j\leq k-1$). If $j=0$, then $\left|V_2\right|=1=2^0$. If $j=1$, then $G[V_2]$ is a monochromatic complete subgraph. Suppose $\left|V_2\right|\geq 3$. Then $G$ contains a 2-colored $K_5$ since $1\in C(V_1)$, a contradiction. Thus $\left|V_2\right|=2=2^1$. If $2\leq j\leq k-1$, then $\left|V_2\right|\leq 2^{j}$ by the induction hypothesis.

(3) By symmetry, we only prove it for $A$. If $\left|C(A)\right|\neq 1$, then $\left|A\right|\leq 2^{\left|C(A)\right|}$ by the same argument as in (2). If $\left|C(A)\right|=1$, then $G[A]$ is a monochromatic complete subgraph. Suppose $\left|A\right|\geq 3$, say $u, v, w\in A$. Recall that $1\notin C(A)$. By the definition of $A$ and $B$, there exists a vertex $x\in B$ such that $c(ux)=1$. In order to avoid a rainbow $K_3$, we have $C(x, \{v, w\})\subseteq C(A)\cup\{1\}$. Note that for any vertex $y\in V_2$, we have $c(y, \{u, v, w, x\})=1$. Thus $\{u,v,w,x,y\}$ forms a 2-colored $K_5$, a contradiction.
\end{proof}

\noindent\begin{claim}\label{cl:p=5 4} $\left|A\right|\geq 2$ and $\left|B\right|\geq 2$.
\end{claim}

\begin{proof} By symmetry, we only prove $\left|A\right|\geq 2$. If $\left|A\right|=1$, say $A=\{u\}$, then we have $c(u, B)=1$ by the definition of $B$. By Claim \ref{cl:p=5 3} (1), we have $C(B)\cap C(V_2)=\emptyset$, so $\left|C(B)\right|+\left|C(V_2)\right|\leq k-1$. If $V''_1=\emptyset$, then by Claim \ref{cl:p=5 3} (2) and (3) we have $\left|V(G)\right|= \left|A\right|+\left|B\right|+\left|V_2\right|\leq 1+2^{\left|C(B)\right|}+2^{\left|C(V_2)\right|}\leq 1+(2^{\left|C(B)\right|+\left|C(V_2)\right|}+1)\leq 1+2^{k-1}+1<2^k+1$, a contradiction. Thus $V''_1\neq \emptyset$, say $v\in V''_1$. Note that $v$ is not incident with any edge in color 1. Thus we may further assume that $c(uv)=2$. In order to avoid a rainbow $K_3$, we have $c(v, B)=2$. Then $2\notin C(B)$ and $2\notin C(V_2)$ in order to avoid a 2-colored $K_5$, so $\left|C(B)\right|+\left|C(V_2)\right|\leq k-2$. Then $\left|V(G)\right|= \left|A\right|+\left|B\right|+\left|V_2\right|+\left|V''_1\right|\leq 1+2^{\left|C(B)\right|}+2^{\left|C(V_2)\right|}+2^{k-1}\leq 1+(2^{\left|C(B)\right|+\left|C(V_2)\right|}+1)+2^{k-1}\leq 1+2^{k-2}+1+2^{k-1}<2^k+1$, a contradiction.
\end{proof}

\noindent\begin{claim}\label{cl:p=5 5} $C(B)\setminus C(A)\neq \emptyset$ and $C(A)\setminus C(B)\neq \emptyset$.
\end{claim}

\begin{proof} By symmetry, we only prove $C(B)\setminus C(A)\neq \emptyset$. Recall that $1\notin C(A)$ and $1\notin C(B)$. By Claim \ref{cl:p=5 4}, we have $\left|C(B)\right|\geq 1$. Since $G[B]$ is a Gallai-coloring, there exists a color, say color 2, inducing a connected spanning subgraph of $G[B]$ by Corollary \ref{co:Gallai}. We will show that $2\notin C(A)$. For a contradiction, suppose that there are two vertices $u, v\in A$ with $c(uv)=2$. We may further assume that $c(uw)=1$ and $c(wx)=2$, where $w,x\in B$. Then $c(vw)\in \{1, 2\}$ and $c(ux)\in \{1, 2\}$. Thus $c(vx)\notin \{1, 2\}$, since otherwise $\{u, v, w, x\}$ together with a vertex in $V_2$ forms a 2-colored $K_5$. Then $c(vw)=c(ux)=2$ in order to avoid a rainbow $K_3$. Since $v$ is incident with some edge in color 1, we may assume that $c(vy)=1$ for some $y\in B\setminus \{w, x\}$. In order to avoid a rainbow $K_3$, we have $C(y, \{u,w\})\subseteq \{1, 2\}$. Then $\{u, v, w, y\}$ together with a vertex in $V_2$ forms a 2-colored $K_5$, a contradiction.
\end{proof}

By Claim \ref{cl:p=5 3} (1), we have $C(A)\cap C(V_2)=\emptyset$ and $C(B)\cap C(V_2)=\emptyset$. Recall that $1\notin C(A)$ and $1\notin C(B\cup V''_1)$. By Claim \ref{cl:p=5 5}, we further have $\left|C(A)\right|\leq k-\left|\{1\}\right|-\left|C(B)\setminus C(A)\right|-\left|C(V_2)\right|\leq k-2-\left|C(V_2)\right|$. Thus $\left|V(G)\right|= \left|A\right|+\left|V_2\right|+\left|B\cup V''_1\right|\leq  2^{k-2-\left|C(V_2)\right|}+2^{\left|C(V_2)\right|}+2^{k-1}\leq 2^{k-2}+1+2^{k-1}<2^k+1$, a contradiction. This completes the proof of Theorem \ref{th:p=5}.
\end{proof}

Next, we present our proof of Theorem~\ref{th:q=log}.

\begin{proof}[Proof of Theorem~\ref{th:q=log}] Let $q=\left\lfloor \log_2 (p-1)\right\rfloor$. By Theorem \ref{th:p=5}, we have $g^q_2(5)=2^q+1\leq p$. Thus every Gallai-$q$-colored $K_p$ contains a 2-colored $K_5$. Let $g=g^k_q(p)$. Then every Gallai-$k$-coloring of $K_{g}$ contains a Gallai-$q$-colored $K_p$, and thus a 2-colored $K_5$. Hence, $g^k_q(p)\geq g^k_2(5)=2^k+1$.
\end{proof}

In fact, we can generalize Theorem \ref{th:q=log} as follows.

\noindent\begin{theorem}\label{th:q=log+}
For integers $p, q, k$ with $q\leq \log_2 (p-1)$ and $k\geq q$, we have $g^k_q(p)\geq \left\lfloor (p-1)^{1/q}\right\rfloor^k+1$.
\end{theorem}

\begin{proof} Let $m=\left\lfloor (p-1)^{1/q}\right\rfloor$.
We show that there is a Gallai-$k$-coloring of $K_{m^k}$, in which there is no $K_p$ receiving at most $q$ distinct colors. Let $G_1$ be a monochromatic copy of $K_{m}$ with color 1. Suppose for some $1\leq i\leq k-1$ we have constructed a Gallai-$i$-coloring $G_i$ of $K_{m^i}$. Then we construct an $(i+1)$-edge-coloring $G_{i+1}$ of $K_{m^{i+1}}$ by joining $m$ copies of $G_{i}$ using edges that all get color $i+1$. Finally, we obtain a Gallai-$k$-coloring $G_k$ of $K_{m^k}$. For any $q$ distinct colors $c_1, c_2, \ldots, c_q$, the largest complete subgraph in $G_k$ using only these $q$ colors has order at most $m^q\leq p-1$. Thus $g^k_{q}(p)\geq m^k+1$.
\end{proof}

By Theorems~\ref{th:q} and \ref{th:q=log+}, we have $\lfloor(p-1)^{1/q}\rfloor^k+1 \leq g^k_q(p) \leq 2^{\frac{2k(p-2)}{q}+1}$ for $2\leq q\leq \log_2(p-1)$. In the case $q=1$, we have $2^{\left(\frac{1}{4}+o(1)\right)kp}\leq g^k_1(p)\leq 2^{(2-o(1))kp}$, where the lower bound follows from a construction given in \cite{FoGP} and the upper bound follows from Proposition~\ref{prop:Kp}. In this case, when $p$ is large, the gap between the lower and upper bounds is much smaller than the gap between the abovementioned lower and upper bounds. In the case $p=5$ and $q=2$, our Theorem~\ref{th:p=5} shows that $g^k_2(5)=2^k+1$ for $k\geq 2$, which closes the gap in this case.

\section{Proof of Theorem \ref{th:q=k-1}}
\label{sec:ch-proof q=k-1}

We first introduce two additional definitions and prove some useful lemmas.
An {\it exact Gallai-$k$-coloring} is a Gallai-$k$-coloring in which all the $k$ colors are used. A {\it star} of order $t+1\ge 2$  is a connected graph with a vertex of degree $t$ having $t$ neighbors of degree 1, and is usually denoted as  $K_{1,t}$.

\noindent\begin{lemma}\label{le:Gallai-n-1}
For any $n\geq 2$, there are at least $\left\lceil\frac{n}{2}\right\rceil$ colors each inducing a star in every exact Gallai-$(n-1)$-coloring of $K_{n}$.
\end{lemma}

\begin{proof} We prove the statement by induction on $n$. For the base case, if $2\leq n\leq 3$, the statement clearly holds. Now assume that it holds for every $2\leq n'\leq n-1$, and we will prove it for $n$. Let $G$ be an exact Gallai-$(n-1)$-coloring of $K_{n}$. Let $V_1, V_2, \ldots, V_m$ be a Gallai-partition of $V(G)$ such that $m$ is minimum. If $m\geq 4$, then by Theorem \ref{th:antiR} we have $|C(G)|\leq 2+\sum^m_{i=1}|C(V_i)|\leq 2+\sum^m_{i=1}(|V_i|-1)\leq 2+n-m\leq n-2$, a contradiction. Thus $2\leq m\leq 3$, so $m=2$ by the minimality of $m$. Without loss of generality, we may assume that $c(V_1, V_2)=1$ and $|V_1|\geq |V_2|$. We claim that $1\notin C(V_1)\cup C(V_2)$ and $C(V_1)\cap C(V_2)= \emptyset$, since otherwise $|C(G)|\leq 1+|C(V_1)|+|C(V_2)|-1\leq |V_1|-1+|V_2|-1\leq n-2$. If $|V_2|=1$, then $G[V_1]$ is an exact Gallai-$(n-2)$-coloring of $K_{n-1}$. By the induction hypothesis, the number of colors each inducing a star is at least $1+\lceil (n-1)/2\rceil\geq\lceil n/2 \rceil$. If $|V_1|\geq |V_2|\geq 2$, then by the induction hypothesis, the number of colors each inducing a star is at least $\lceil |V_1|/2\rceil+\lceil |V_2|/2\rceil\geq \lceil n/2\rceil$.
\end{proof}

\noindent\begin{lemma}\label{le:N} For integers $c, n, N$ with $c\geq 2$, $n\geq 2(7+c)^c$ and $N\geq n-\frac{3n}{(7+c)^{c}}$, we have $\frac{N}{(2+c)(6+c)^{c-1}}-2\geq \frac{n}{(7+c)^c}$.
\end{lemma}

\begin{proof} Let $a=6+c$. Then $a\geq 8$. Since
\begin{align*}
 ~ &~(2+c)(6+c)^{c-1}\left(\frac{N}{(2+c)(6+c)^{c-1}}-2- \frac{n}{(7+c)^c}\right)\\
 \geq &~\left(n-\frac{3n}{(7+c)^{c}}\right)-2(2+c)(6+c)^{c-1}- \frac{(2+c)(6+c)^{c-1}n}{(7+c)^c}\\
 = &~\left(1-\frac{3}{(7+c)^{c}}-\frac{(2+c)(6+c)^{c-1}}{(7+c)^c}\right)n-2(2+c)(6+c)^{c-1}\\
 \geq &~\left(1-\frac{3}{(7+c)^{c}}-\frac{(2+c)(6+c)^{c-1}}{(7+c)^c}\right)2(7+c)^c-2(2+c)(6+c)^{c-1}\\
 = &~2(7+c)^c-4(2+c)(6+c)^{c-1}-6\\
 = &~2((a+1)^{a-6}-2(a-4)a^{a-7})-6\\
 = &~2\left(\sum^{a-6}_{i=0}\binom{a-6}{i}a^i-2a^{a-6}+8a^{a-7}\right)-6\\
 = &~2\left(\sum^{a-8}_{i=0}\binom{a-6}{i}a^i+(a-6)a^{a-7}+a^{a-6}-2a^{a-6}+8a^{a-7}\right)-6\\
 = &~2\sum^{a-8}_{i=0}\binom{a-6}{i}a^i+4a^{a-7}-6~\geq~0,
\end{align*}
we have $\frac{N}{(2+c)(6+c)^{c-1}}-2\geq \frac{n}{(7+c)^c}$.
\end{proof}

\noindent\begin{lemma}\label{le:Gallai-n-c}
For any $c\geq 1$ and $n\geq 2(7+c)^c$, there are at least $\frac{n}{(7+c)^c}$ colors each inducing a star in every exact Gallai-$(n-c)$-coloring of $K_{n}$.
\end{lemma}

\begin{proof} We prove the statement by induction on $c$. For the base case, if $c=1$, the statement holds by Lemma \ref{le:Gallai-n-1}. Now assume that it holds for every $1\leq c'\leq c-1$, and we will prove it for $c$ with $c\geq 2$. Let $G$ be an exact Gallai-$(n-c)$-coloring of $K_{n}$ using colors $1, 2, \ldots, n-c$. For a contradiction, suppose that the number of colors each inducing a star in $G$ is less than $\frac{n}{(7+c)^c}$.

\noindent\begin{claim}\label{cl:Gallai-n-c-1} Let $N$ be an integer satisfying $N\leq n$ and $\frac{N}{(2+c)(6+c)^{c-1}}-2\geq \frac{n}{(7+c)^c}$. For any $V'\subseteq V(G)$ with $|V'|=N$, $C(V')\cap C(V', V(G)\setminus V')=\emptyset$ and $C(V')\cap C(V(G)\setminus V')=\emptyset$, let $G'=G[V']$. If $G'$ is an exact Gallai-$(N-c)$-coloring of $K_{N}$, then $V(G')$ has a Gallai-partition consisting of exactly two parts $V'_1$ and $V'_2$, such that $|C(V'_1)|=|V'_1|-c$, $|C(V'_2)|=|V'_2|-1$, $c(V'_1, V'_2)\notin C(V'_1)\cup C(V'_2)$ and $C(V'_1)\cap C(V'_2)=\emptyset$.
\end{claim}

\begin{proof} Note that the integer $N$ and subset $V'$ satisfying the above conditions exist since we can choose $N=n$, $V'=V(G)$ and $G'=G$. Without loss of generality, let $C(G')=[N-c]$. First, we assume that there exists some color $\ell\in [N-c]$ such that the subgraph of $G'$ induced by color $\ell$ has at least two nontrivial components. Then we recolor all the edges of color $\ell$ in one of its nontrivial components with color $N-c+1$. Let $G''$ be the resulting coloring of $K_N$. It is easy to check that $G''$ is an exact Gallai-$(N-(c-1))$-coloring of $K_{N}$. Since $\frac{N}{(2+c)(6+c)^{c-1}}-2\geq \frac{n}{(7+c)^c}$, we have $N\geq 2(7+(c-1))^{c-1}$. By the induction hypothesis, there are at least $\frac{N}{(7+(c-1))^{c-1}}$ colors each inducing a star in $G''$. Recall that $C(V')\cap C(V', V(G)\setminus V')=\emptyset$ and $C(V')\cap C(V(G)\setminus V')=\emptyset$. There are at least $\frac{N}{(7+(c-1))^{c-1}}-2\geq \frac{n}{(7+c)^c}$ colors each inducing a star in $G$, a contradiction.

Next, we may assume that every color induces a subgraph with exactly one nontrivial component in $G'$. Let $V'_1, V'_2, \ldots, V'_m$ be a Gallai-partition of $V(G')$ such that $m$ is minimum and $|V'_1|=\max_{1\leq i\leq m}\{|V'_i|\}$, and let $S$ be the set of colors used between these parts. Then $1\leq |S|\leq 2$ and $(C(V'_i)\cap C(V'_j))\setminus S=\emptyset$ for every $1\leq i<j\leq m$.

If $m\geq 4$, then $N-c=|C(G')|\leq |S|+\sum^m_{i=1}|C(V'_i)|\leq 2+\sum^m_{i=1}(|V'_i|-1)\leq N+2-m$, so $m\leq 2+c$. Thus $|V'_1|\geq N/(2+c)\geq 2(6+c)^{c-1}= 2(7+(c-1))^{c-1}$. Moreover, $|C(V'_1)|\geq N-c-|S|-\sum^m_{i=2}|C(V'_i)|\geq N-c-2-\sum^m_{i=2}(|V'_i|-1)=|V'_1|-c+m-3\geq |V'_1|-c+1$. Let $C(V'_1)=\{c_i\colon\, i=1, 2, \ldots, |C(V'_1)|\}$. For all $|V'_1|-c+2\leq j\leq |C(V'_1)|$ (if $|C(V'_1)|>|V'_1|-c+1$), we recolor all the edges of color $c_j$ with color $c_1$ in $G'[V'_1]$, so we obtain an exact Gallai-$(|V'_1|-(c-1))$-coloring of $K_{|V'_1|}$. By the induction hypothesis, there are at least $\frac{|V'_1|}{(7+(c-1))^{c-1}}$ colors each inducing a star in $G'[V'_1]$. Thus the number of colors each inducing a star in $G$ is at least $\frac{|V'_1|}{(7+(c-1))^{c-1}}-|S|\geq \frac{N}{(2+c)(7+(c-1))^{c-1}}-2\geq \frac{n}{(7+c)^c}$, a contradiction.

Thus $2\leq m\leq 3$, so $m=2$ by the minimality of $m$. Then $|V'_1|\geq N/2\geq 2(7+(c-1))^{c-1}$. Note that $|C(V'_1)|\geq N-c-|C(V'_2)|-|S|\geq N-c-|V'_2|+1-1=|V'_1|-c$. If $|C(V'_1)|\geq |V'_1|-c+1$, then we can derive a contradiction by a similar argument as above. Thus we have $|C(V'_1)|=|V'_1|-c$, so $|C(V'_2)|=|V'_2|-1$, $S\cap (C(V'_1)\cup C(V'_2))=\emptyset$ and $C(V'_1)\cap C(V'_2)=\emptyset$.
\end{proof}

We will use an algorithm to find $\lceil\frac{n}{(7+c)^c}\rceil$ colors each inducing a star in $G$. Let $V^{(0)}_1:= V(G)$, $V^{(0)}_2:=\emptyset$, $G^{(0)}:=G$, $t:=1$, $A:=\emptyset$ and $B:=\emptyset$. The algorithm at time $i\geq 1$ consists of two steps.

{\it Step 1.} By applying Claim \ref{cl:Gallai-n-c-1} to $N=|V^{(i-1)}_1|$, $V'=V^{(i-1)}_1$ and $G'=G[V^{(i-1)}_1]$, we obtain a Gallai-partition $V^{(i)}_1$, $V^{(i)}_2$ of $V(G')$ such that $|C(V^{(i)}_1)|=|V^{(i)}_1|-c$, $|C(V^{(i)}_2)|=|V^{(i)}_2|-1$, $c(V^{(i)}_1, V^{(i)}_2)\notin C(V^{(i)}_1)\cup C(V^{(i)}_2)$ and $C(V^{(i)}_1)\cap C(V^{(i)}_2)=\emptyset$.

{\it Step 2.} If $|V^{(i)}_2|=1$, then let $c_t=c(V^{(i)}_1, V^{(i)}_2)$, $t:=t+1$ and $A:=A\cup V^{(i)}_2$; otherwise if $|V^{(i)}_2|\geq 2$, then let $B:=B\cup V^{(i)}_2$.

We repeat the above steps until $t\geq \frac{n}{(7+c)^c} +1$. Finally, we obtain $t-1$ distinct colors $c_1, c_2, \ldots, c_{t-1}$ each inducing a star in $G$. It remains to show that the above algorithm is valid. Since for any $j\leq i-1$ with $V^{(j)}_2\subseteq B$ we have $|V^{(j)}_2|\geq 2$ and $|C(V^{(j)}_2)|=|V^{(j)}_2|-1$, the number of colors each inducing a star in $G[V^{(j)}_2]$ is at least $\lceil |V^{(j)}_2|/2\rceil$ by Lemma \ref{le:Gallai-n-1}. Recall that $c(V^{(j)}_1, V^{(j)}_2)\notin C(V^{(j)}_1)\cup C(V^{(j)}_2)$ and $C(V^{(j)}_1)\cap C(V^{(j)}_2)=\emptyset$ for every $j\leq i-1$. Thus $|B|<\frac{2n}{(7+c)^c}$; otherwise the number of colors each inducing a star in $G$ is at least $\frac{n}{(7+c)^c}$, a contradiction. Thus $|V^{(i-1)}_1|= n-|B|-|A|> n-\frac{3n}{(7+c)^c}$. By Lemma \ref{le:N}, $|V^{(i-1)}_1|$ satisfies the condition of $N$ in Claim \ref{cl:Gallai-n-c-1}. Moreover, $V^{(i-1)}_1$ (resp., $G^{(i-1)}$) satisfies the condition of $V'$ (resp., $G'$) in Claim \ref{cl:Gallai-n-c-1}. Thus we can apply Claim \ref{cl:Gallai-n-c-1} in Step 1, so the algorithm is valid.
\end{proof}

Now we have all ingredients to present our proof of Theorem~\ref{th:q=k-1}.

\begin{proof}[Proof of Theorem~\ref{th:q=k-1}] The cases $c\in\{1, 2\}$ follow from Theorems \ref{th:q=p-2} and \ref{th:q=p-3}, so we may assume that $c\geq 3$. The lower bound $g^k_{k-1}(p)>p$ is trivial. For the upper bound, let $G$ be a Gallai-$k$-coloring of $K_{p+1}$. We may assume that $G$ is an exact Gallai-$k$-coloring, where $k=p-c=p+1-(c+1)$. By Lemma \ref{le:Gallai-n-c}, the number of colors each inducing a star in $G$ is at least $\frac{p+1}{(7+(c+1))^{c+1}}\geq 2$. Let $i$ be a color that induces a star in $G$, and let $v$ be a vertex with maximum degree in this star. Then $G-v$ is a copy of $K_p$ using at most $k-1$ colors. The result follows.
\end{proof}

\section{Concluding remarks}
\label{sec:ch-remark}

In this paper, we studied the behavior of $g(n,p,q)$, which is the minimum number of colors that are needed for $K_n$ to have a Gallai-coloring in which every $K_p$ receives at least $q$ distinct colors. For this purpose it was convenient to consider the closely related function $g^k_{q}(p)$. We now recapitulate what the above results on $g^k_{q}(p)$ imply for the function $g(n,p,q)$.

Corollary \ref{co:q=p-1} implies that $g(n,p,q)$ makes sense only for $2\leq q\leq p-1$. Theorem \ref{th:q} implies that $g(n,p,q)>\frac{q-1}{2(p-2)}(\log_2 n -1)$. For appropriate $p$ and $n$, Theorems \ref{th:q=p-2}, \ref{th:q=p-3}, \ref{th:p=5}, \ref{th:q=sqrt} and \ref{th:q=log} imply that $g(n,p,p-1)=n-1$, $g(n,p,p-2)=n-2$, $g(n,5,3)=\lceil\log_2 n\rceil$, $g(n,p,\lfloor\sqrt{p-1}\rfloor)\leq \left\lceil\sqrt{n}\right\rceil-1$ and $g(n,p,\lfloor\log_2(p-1)\rfloor+1)\leq \lceil\log_2 n\rceil$, respectively.

We remark that the behavior of $g(n,p,q)$ is very different from $f(n,p,q)$, as may be seen by noting that in the case $q=p-1$, Conlon et al. \cite{CFLS2} proved that $f(n,p,p-1)$ is subpolynomial in $n$, but here we show that $g(n,p,p-1)=n-1$.  A natural problem is to find the threshold for linear $g(n, p, q)$, i.e., the smallest $q$ such that $g(n, p, q)$ is linear in $n$. We were not able to solve this problem, but in light of Theorems \ref{th:q=p-2}, \ref{th:q=p-3} and \ref{th:q=k-1}, we conjecture the following.

\noindent\begin{conjecture}\label{conj:c} For any constant $c\geq 2$, there exists a $p_0$ such that for all integers $p\geq p_0$ and $k\geq p-c$, we have $g^{k}_{p-c}(p)=k+c$. {\rm (}Equivalently, for any constant $c'\geq 1$, there exists a $p_0$ such that for all integers $p\geq p_0$ and $n\geq p$, we have $g(n,p,p-c')=n-c'$.{\rm )}
\end{conjecture}

The following construction shows that $g^k_{p-c}(p)\geq k+c$. Let $G$ be a copy of $K_{k+c-1}$ with vertex set $\{v_1, v_2, \ldots, v_{k+c-1}\}$. For every $1\leq i\leq k$ and $i< j\leq k+c-1$, we color the edge $v_iv_j$ using color $i$, and we color all the remaining edges with color $k$. In the case $c\in \{2, 3\}$, Theorems~\ref{th:q=p-2} and \ref{th:q=p-3} confirm Conjecture~\ref{conj:c}. For $c\geq 4$, Theorem~\ref{th:q=k-1} shows that $g^k_{p-c}(p)=k+c$ for $k=p-c+1$ and sufficiently large $p$.

Theorems \ref{th:q=sqrt} and \ref{th:q=log} imply that $g(n, p, \lfloor\sqrt{p-1}\rfloor)$ and $g(n,p,\lfloor\log_2(p-1)\rfloor+1)$ are at most $O(n^{1/2})$ and $O(\log n)$, respectively. We know that $g(n,p,p-1)$ is linear in $n$, $g(n, p, 2)$ is logarithmic in $n$, and $g(n, p, q)\geq g(n, p, q-1)$. Thus for any fixed $p$ (where $p$ is large enough), there exists a value $q$ such that $g(n, p, q)$ is polynomial in $n$ and $g(n, p, q-1)$ is subpolynomial in $n$. Another natural problem is to find the smallest $q$ such that $g(n, p, q)=\Theta(n^c)$ for some constant $0<c<1$, and the largest $q$ such that $g(n, p, q)=\Theta(\log n)$.

Recently, Krueger \cite{Kru} studied the minimum number of colors in an edge-coloring of $K_n$ such that every $P_m$ receives at least $q$ colors. For a general fixed graph $H$, one can also study the minimum number of colors in a Gallai-coloring of $K_n$ such that every $H$ receives at least $q$ colors. This problem generalizes the concept of Gallai-Ramsey numbers of graphs. For recent results on Gallai-Ramsey theory, we refer the interested reader to \cite{LSSZ,LMSSS,LiCh}.

\section*{Acknowledgement}

The authors are grateful to the anonymous referees for valuable comments, suggestions and corrections which improved the presentation of this paper.

\appendix

\section*{Appendix}

\section{Proof of $g^4_3(6)=8$}
\label{sec:A}

We first show that $g^4_3(6)>7$ by construction. Taking a copy of $K_7$ with vertex set $U\cup \{x, y, z\}$, where $U=\{u, v, w, s\}$, we color the edges such that $c(uv)=c(vw)=c(ws)=1$, $c(vs)=c(su)=c(uw)=2$, $c(x, U)=c(U, z)=c(zy)=3$ and $c(U, y)=c(yx)=c(xz)=4$. It is easy to check that the resulting coloring is a Gallai-4-coloring of $K_7$ without a 3-colored $K_6$.

Next we prove that $g^4_3(6)\leq 8$. For a contradiction, suppose that $G$ is a Gallai-4-colored $K_8$ containing no 3-colored $K_6$ and $V(G)=\{u_1, u_2, \ldots, u_8\}$. For the proof of Claim \ref{cl:A_1} below, we will use K\"{o}nig's Theorem \cite{Kon} which states that the size of a minimum covering is the same as  the size of a maximum matching in a bipartite graph, where a covering of a graph $F$ is a subset $V\subseteq V(F)$ such that every edge of $F$ has at least one end in $V$.

\begin{claim}\label{cl:A_1} There is no vertex $u\in V(G)$ such that $|C(u, V(G-u))|=1$.
\end{claim}

\begin{proof} Suppose that there exists a vertex $u\in V(G)$ such that $|C(u, V(G-u))|=1$, say $c(u_8, V(G-u_8))=1$. Let $F$ be the spanning subgraph of $G-u_8$ consisting of all the edges in color 1. Let $S\subseteq V(G-u_8)$ be a minimum covering of $F$ and $T=V(F)\setminus S=V(G)\setminus (S\cup \{u_8\})$. In order to avoid a 3-colored $K_6$, we have $\left|S\right|\geq 2$.

We next show that $F$ is a bipartite graph. First, $F$ contains no $C_3$; otherwise $G$ contains a monochromatic $K_4$ in color 1, so there is a 3-colored $K_6$ by Lemma \ref{le:Gallai}. Second, suppose that $W$ is a copy of $C_5$ in $F$. It is easy to check that $|C(V(W))\cap \{2,3,4\}|\leq 2$ in order to avoid a rainbow triangle, so $W\cup \{u_8\}$ forms a 3-colored $K_6$, a contradiction. Thus $F$ contains no $C_5$. Finally, suppose that $F$ contains a copy of $C_7$, say $u_1u_2\cdots u_7u_1$. Since $F$ contains no $C_3$, we may assume that $c(u_1u_3)=2$ without loss of generality. In order to avoid a rainbow triangle in $G$ or a $C_5$ in $F$, we have $c(u_1u_4)=2$. Then $c(u_2u_4)=2$ in order to avoid a rainbow triangle in $G$ or a $C_3$ in $F$. Now $\{u_8, u_1, u_2, u_3, u_4\}$ forms a 2-colored $K_5$. By Lemma \ref{le:Gallai}, there is a 3-colored $K_6$ in $G$. This contradiction implies that $F$ contains no odd cycle and thus $F$ is a bipartite graph.

Since $F$ is bipartite, we have $|S|\leq \left\lfloor\left|V(G-u_8)\right|/2\right\rfloor=3$. If $\left|S\right|=3$, then $\left|T\right|=4$, say $S=\{u_1, u_2, u_3\}$ and $T=\{u_4, u_5, u_6, u_7\}$. By K\"{o}nig's Theorem, there is a matching of size 3 in $F$, i.e., $G-u_8$ contains three pairwise nonadjacent edges in color 1. It is easy to check that these three edges must appear  between $S$ and $T$. Without loss of generality, we may assume that $c(u_1u_4)=c(u_2u_5)=c(u_3u_6)=1$. Note that $1\notin C(T)$, so we may further assume that $C(\{u_4, u_5, u_6\})\subseteq \{2, 3\}$. Then $C(S, \{u_4, u_5, u_6\})\subseteq \{1, 2, 3\}$ in order to avoid a rainbow triangle. Now $S$ forms a monochromatic $K_3$ in color 4; otherwise there is a 3-colored $K_6$ within $S\cup \{u_8, u_4, u_5, u_6\}$. Since $G$ is a Gallai-coloring, it is easy to see that $c(S, \{u_4, u_5, u_6\})=1$. Then $S\cup \{u_8, u_4, u_5\}$ forms a 3-colored $K_6$, a contradiction. Therefore, we have $\left|S\right|=2$ and $\left|T\right|=5$, say $S=\{u_1, u_2\}$ and $T=\{u_3, u_4, \ldots, u_7\}$. Note that $C(T)\subseteq \{2, 3, 4\}$, so there is a 2-colored $K_4$ within $T$ by Theorem \ref{th:q=p-2}. Without loss of generality, we may assume that $C(T')\subseteq\{2, 3\}$, where $T'=\{u_3, u_4, u_5, u_6\}$. Moreover, we may assume that $c(u_1u_3)=1$ by K\"{o}nig's Theorem. Then $C(T'\cup \{u_1\})=\{1, 2, 3\}$ by Lemma \ref{le:Gallai}, which implies that $T'\cup \{u_8, u_1\}$ forms a 3-colored $K_6$, a contradiction.
\end{proof}

Let $V_1, V_2, \ldots, V_m$ be a Gallai-partition of $V(G)$ such that $m$ is minimum. If there is only one color between these parts, then we assume that color 1 is this color, and if there are two colors between these parts, then we assume that colors 1 and 2 are these two colors.

\begin{claim}\label{cl:A_2} $m\leq 3$.
\end{claim}

\begin{proof} Suppose $m\geq 4$. In order to avoid a 3-colored $K_6$, the following statements hold: (1) $m=4$; (2) $1, 2\notin \bigcup^{m}_{i=1}C(V_i)$; (3) color 3 (resp., color 4) is used in exactly one of these parts; (4) there is neither a 2-colored $K_5$ nor a monochromatic $K_3$ within each part. Thus we have $\sum^{m}_{i=1}\left|V_i\right|\leq \max\{4+1+1+1, 2+2+1+1\}=7< \left|V(G)\right|$, a contradiction.
\end{proof}

By Claim \ref{cl:A_2} and the minimality of $m$, we have $m=2$. Note that at most one of $1\in C(V_1)$ and $1\in C(V_2)$ holds; otherwise $G$ contains a monochromatic $K_4$ and thus $G$ contains a 3-colored $K_6$ by Lemma \ref{le:Gallai}. First, we consider the case that $1\notin C(V_1)$ and $1\notin C(V_2)$, i.e., $C(V_i)\subseteq \{2, 3, 4\}$ for $i\in [2]$. In this case, we have $\left|V_1\right|\leq 5$ and $\left|V_2\right|\leq 5$, so $\left|V_1\right|\geq 3$ and $\left|V_2\right|\geq 3$. We claim that there is no 2-colored $K_4$ within each $V_i$ for $i\in [2]$. Indeed, if there is a 2-colored copy $K$ of $K_4$ in some $V_i$, say $i=1$ and $C(K)\subseteq \{2, 3\}$, then $C(V_2)=\{4\}$ in order to avoid a 3-colored $K_6$. If $\left|V_2\right|=4$, then it is easy to find a 3-colored $K_6$ in $G$. If $\left|V_2\right|=3$, then $\left|V_1\right|=5$. Now $G$ also contains a 3-colored $K_6$ no matter whether $4\in C(V_1)$ or $4\notin C(V_1)$. Hence, there is no 2-colored $K_4$ within each $V_i$ for $i\in [2]$. By Theorem \ref{th:q=p-2}, we have $g^3_2(4)=5$, so $\left|V_i\right|\leq 4$ for $i\in [2]$, which implies that $\left|V_1\right|=\left|V_2\right|=4$ and moreover, $C(V_1)=C(V_2)=\{2, 3, 4\}$. For any two distinct colors $c_1, c_2\in \{2, 3, 4\}$, we say that $c_1$ and $c_2$ are adjacent if there are two adjacent edges $e$ and $f$ such that $c(e)=c_1$ and $c(f)=c_2$. Note that if $c_1$ and $c_2$ are adjacent, then there is a 2-colored $K_3$ with colors $c_1$ and $c_2$. Without loss of generality, we may assume that colors 2 and 3 are adjacent in $G[V_1]$.  Then colors 2 and 3 are nonadjacent in $G[V_2]$, and thus color 4 is adjacent to both color 2 and color 3 in $G[V_2]$. But then color 4 is adjacent to neither color 2 nor color 3 in $G[V_1]$, a contradiction.

Finally, we consider the case that exactly one of $1\in C(V_1)$ and $1\in C(V_2)$ holds, say $1\in C(V_1)$ and $1\notin C(V_2)$.

\begin{claim}\label{cl:A_3} The following statements hold:
\begin{itemize}
  \item[{\rm (1)}] there is no monochromatic $K_3$ in color 1 within $V_1$;
  \item[{\rm (2)}] $C(V_1)\cap C(V_2)=\emptyset$.
\end{itemize}
\end{claim}

\begin{proof} If (1) does not hold, then $G$ contains a monochromatic $K_4$. By Lemma \ref{le:Gallai}, there is a 3-colored $K_6$. This contradiction proves (1). We next prove (2). Suppose that $C(V_1)\cap C(V_2)\neq \emptyset$, say $2\in C(V_1)\cap C(V_2)$. If colors 1 and 2 are adjacent in $G[V_1]$, then $G$ contains a 2-colored $K_5$. By Lemma \ref{le:Gallai}, there is a 3-colored $K_6$, a contradiction. Thus colors 1 and 2 are nonadjacent in $G[V_1]$, and we may assume that $u_1, u_2, u_3, u_4\in V_1$ such that $c(u_1u_2)=1$, $c(u_3u_4)=2$ and $1, 2\notin C(\{u_1, u_2\}, \{u_3, u_4\})$. Since $G$ is a Gallai-coloring, we have $\left|C(\{u_1, u_2\}, \{u_3, u_4\})\right|=1$. Assume that $u_5u_6$ is an edge with color 2 in $G[V_2]$. Then $\{u_1, u_2, \ldots, u_6\}$ forms a 3-colored $K_6$, a contradiction. This completes the proof of (2).
\end{proof}

In order to avoid a 3-colored $K_6$, we have $\left|V_2\right|\leq 5$, so $\left|V_1\right|\geq 3$. Then $\left|C(V_1)\right|\geq 2$ by Claim \ref{cl:A_3} (1).
If $\left|C(V_1)\right|=2$, then $\left|C(V_2)\right|=2$ by Claim \ref{cl:A_3} (2). In this case, we have $\left|V_1\right|\leq 3$ (resp., $\left|V_2\right|\leq 4$) in order to avoid a 3-colored $K_6$. Then $\left|V_1\right|+\left|V_2\right|\leq 3+4<\left|V(G)\right|$, a contradiction.
If $\left|C(V_1)\right|=3$, then $\left|C(V_2)\right|=1$ by Claim \ref{cl:A_3} (2). In this case, we have $\left|V_1\right|\leq 4$ (resp., $\left|V_2\right|\leq 3$) in order to avoid a 3-colored $K_6$. Then $\left|V_1\right|+\left|V_2\right|\leq 4+3<\left|V(G)\right|$, a contradiction.
If $\left|C(V_1)\right|=4$, then $\left|C(V_2)\right|=0$ by Claim \ref{cl:A_3} (2). In this case, we have $\left|V_2\right|=1$, contradicting to Claim \ref{cl:A_1}. This contradiction completes the proof.

\section{Proof of $g^5_3(6)=10$}
\label{sec:B}

We first show that $g^5_3(6)>9$ by construction. Taking a copy of $K_9$ with vertex set $U\cup V\cup \{x, y\}$, where $U=\{r, s, t\}$ and $V=\{u, v, w, z\}$, we color the edges such that $c(\{x,y\}, U\cup V)=1$, $c(xy)=c(\{u, z\}, \{v, w\})=2$, $c(rs)=c(st)=c(vw)=3$, $c(rt)=c(uz)=4$ and $c(U, V)=5$. Let $G'$ be the resulting coloring. It is easy to check that $G'$ is a Gallai-coloring. Moreover, for any two distinct colors $i,j\in[5]$, we need to delete at least four vertices such that there is neither color $i$ nor color $j$ on edges of the remaining graph. Thus $G'$ is a Gallai-5-coloring of $K_9$ without a 3-colored $K_6$.

Next we prove that $g^5_3(6)\leq 10$. For a contradiction, suppose that $G$ is a Gallai-5-colored $K_{10}$ containing no 3-colored $K_6$ and $V(G)=\{u_1, u_2, \ldots, u_{10}\}$. Let $V_1, V_2, \ldots, V_m$ be a Gallai-partition of $V(G)$ such that $m$ is minimum. If there is only one color between these parts, then we assume that color 1 is this color, and if there are two colors between these parts, then we assume that colors 1 and 2 are these two colors.

If $m\geq 4$, then in order to avoid a 3-colored $K_6$, the following statements hold: (1) $m=4$; (2) $1, 2\notin \bigcup^{m}_{i=1}C(V_i)$; (3) for each $i\in \{3, 4, 5\}$, color $i$ is used in exactly one of these parts; (4) there is neither a 2-colored $K_5$ nor a monochromatic $K_3$ within each part. Thus we have $\sum^{m}_{i=1}\left|V_i\right|\leq \max\{5+3\cdot 1, 4+2+2\cdot 1, 2+2+2+1\}=8< \left|V(G)\right|$, a contradiction. Hence, $2\leq m\leq 3$. By the minimality of $m$, we have $m=2$. Note that at most one of $1\in C(V_1)$ and $1\in C(V_2)$ holds; otherwise $G$ contains a monochromatic $K_4$ and thus $G$ contains a 3-colored $K_6$ by Lemma \ref{le:Gallai}. We divide the rest of the proof into two cases.

\medskip\noindent
{\bf Case 1.} $1\notin C(V_1)$ and $1\notin C(V_2)$.
\vspace{0.05cm}

\noindent
In this case, we have $C(V_i)\subseteq \{2, 3, 4, 5\}$ for $i\in [2]$. Since $g^4_3(6)=8$, we have $\left|V_i\right|\leq 7$ and thus $\left|V_{3-i}\right|\geq 3$ for each $i\in [2]$. We claim that there is no 2-colored $K_4$ within each $V_i$ for $i\in [2]$. Indeed, if there is a 2-colored copy $K$ of $K_4$ in some $V_i$, say $i=1$ and $C(K)\subseteq \{2, 3\}$, then $C(V_2)\subseteq \{4, 5\}$ in order to avoid a 3-colored $K_6$, and thus $\left|V_2\right|\leq 4$ for the same reason. If $\left|V_2\right|=4$, then $4, 5\notin C(V_1)$, which implies that $G[V_1]$ is a 2-colored $K_6$, a contradiction. If $\left|V_2\right|=3$, then $\left|V_1\right|=7$. In order to avoid a 3-colored $K_6$, we have $C(V_1)=\{2, 3, 4, 5\}$, colors 4 and 5 are nonadjacent in $G[V_1]$, and $G[V_1]$ contains no monochromatic $K_3$ in color 4 or 5. If $G[V_1]$ contains a monochromatic $2K_2$ in color 4 and a monochromatic $2K_2$ in color 5, then $\left|V_1\right|\geq 8$, a contradiction. Thus we may assume that $G[V_1]$ contains no monochromatic $2K_2$ in color 4. Then there is a vertex $u\in V_1$ such that $G[V_1\setminus \{u\}]$ contains no edge in color 4, which implies that $G[V_1\setminus \{u\}]$ is a 3-colored $K_6$, a contradiction. Hence, there is no 2-colored $K_4$ within each $V_i$ for $i\in [2]$.

By Theorem \ref{th:q=p-2}, we have $g^4_2(4)=6$, so $\left|V_i\right|\leq 5$ for $i\in [2]$, which implies that $\left|V_1\right|=\left|V_2\right|=5$ and moreover, $C(V_1)=C(V_2)=\{2, 3, 4, 5\}$ since $g^3_2(4)=5$. Without loss of generality, we may assume that colors 2 and 3 are adjacent in $G[V_1]$. Recall that if two distinct colors $c_1$ and $c_2$ are adjacent, then there is a 2-colored $K_3$ with colors $c_1$ and $c_2$. Then colors 2 and 3 are nonadjacent in $G[V_2]$. We may assume that $u_1, u_2, u_3, u_4\in V_2$ such that $c(u_1u_2)=2$ and $c(u_3u_4)=3$. Then $C(\{u_1, u_2\}, \{u_3, u_4\})\subseteq \{4, 5\}$. In order to avoid a rainbow triangle, there is exactly one color on the edges between $\{u_1, u_2\}$ and $\{u_3, u_4\}$, say color 4. Then color 4 is adjacent to neither color 2 nor color 3 in $G[V_1]$, so color 4 is adjacent to color 5 in $G[V_1]$. Since $\left|V_2\right|=5$ and $G[V_2]$ contains a monochromatic $K_{2,2}$ in color 4, we have that colors 4 and 5 are adjacent in $G[V_2]$. Then $G$ contains a 3-colored $K_6$ with colors 1, 4 and 5, a contradiction.

\medskip\noindent
{\bf Case 2.} Exactly one of $1\in C(V_1)$ and $1\in C(V_2)$ holds.
\vspace{0.05cm}

\noindent
Without loss of generality, we may assume that $1\in C(V_1)$ and $1\notin C(V_2)$. By the same argument as in the proof of  Claim \ref{cl:A_3}, we have $C(V_1)\cap C(V_2)=\emptyset$. Thus $\left|C(V_1)\right|+\left|C(V_2)\right|=5$. In order to avoid a 3-colored $K_6$ and since $g^4_3(6)=8$, we have the following inequality:
\begin{equation*}
\begin{aligned}
\left|V_1\right|+\left|V_2\right| &~ \leq
\begin{cases}
2+7, &~ \text{if }\left|C(V_1)\right|=1\text{ and }\left|C(V_2)\right|=4,\\
3+5, &~ \text{if }\left|C(V_1)\right|=2\text{ and }\left|C(V_2)\right|=3,\\
4+4, &~ \text{if }\left|C(V_1)\right|=3\text{ and }\left|C(V_2)\right|=2,\\
6+2, &~ \text{if }\left|C(V_1)\right|=4\text{ and }\left|C(V_2)\right|=1,
\end{cases}
\end{aligned}
\end{equation*}
contradicting  the fact that $\left|V(G)\right|=10$. Thus it suffices to consider the case $\left|C(V_1)\right|=5$ and $\left|C(V_2)\right|=0$. In this case, we have $\left|V_2\right|=1$, say $V_2=\{u_{10}\}$. Let $F$ be the spanning subgraph of $G[V_1]$ consisting of all the edges in color 1. Let $S\subseteq V_1$ be a minimum covering of $F$ and $T=V_1\setminus S$. Note that $C(T)\subseteq \{2, 3, 4, 5\}$. Since $g^4_3(6)=8$, we have $\left|T\right|\leq 7$ and thus $\left|S\right|\geq 2$.

We next show that $F$ is a bipartite graph. By similar arguments as in the proof of Claim \ref{cl:A_1}, there is no $C_{\ell}$ in $F$ for $\ell\in\{3, 5, 7\}$. We now show that $F$ contains no $C_9$. Suppose for a contradiction that $u_1u_2\cdots u_9u_1$ is a cycle in $F$. Since $F$ contains no $C_3$, we may assume that $c(u_1u_3)=2$ without loss of generality. In order to avoid a rainbow triangle in $G$ or a $C_7$ in $F$, we have $c(u_1u_4)=2$. Then $c(u_2u_4)=2$ in order to avoid a rainbow triangle in $G$ or a $C_3$ in $F$. Now $\{u_{10}, u_1, u_2, u_3, u_4\}$ forms a 2-colored $K_5$. By Lemma \ref{le:Gallai}, there is a 3-colored $K_6$ in $G$. This contradiction implies that $F$ contains no odd cycle and thus $F$ is a bipartite graph.

Since $F$ is bipartite, we have $|S|\leq \left\lfloor\left|V_1\right|/2\right\rfloor=4$.

\begin{claim}\label{cl:B_1} $\left|S\right|=2$ and $\left|T\right|=7$.
\end{claim}

\begin{proof} If $\left|S\right|=3$ and $\left|T\right|=6$ (resp., $\left|S\right|=4$ and $\left|T\right|=5$), then $F$ contains a matching of size 3 (resp., 4) by K\"{o}nig's Theorem. It is easy to check that this matching must appear between $S$ and $T$. Without loss of generality, let $u_1, u_2, u_3\in S$, $u_4, u_5, \ldots, u_8\in T$, and $c(u_1u_4)=c(u_2u_5)=c(u_3u_6)=1$. Note that $1\notin C(T)$, so we may further assume that $C(\{u_4, u_5, u_6\})\subseteq \{2, 3\}$. Then $C(\{u_1, u_2, u_3\}, \{u_4, u_5, u_6\})\subseteq \{1, 2, 3\}$ in order to avoid a rainbow triangle. Now $1, 2, 3 \notin C(\{u_1, u_2, u_3\})$; otherwise there is a 3-colored $K_6$ in $G[\{u_1, u_2, \ldots, u_6, u_{10}\}]$. Thus $C(\{u_1, u_2, u_3\})\subseteq \{4, 5\}$. Then $c(\{u_1, u_2, u_3\},$ $\{u_4, u_5, u_6\})=1$ in order to avoid a rainbow triangle. Note that $1\notin C(T)$. If $C(u_7, \{u_4, u_5, u_6\})\cap\{4, 5\}\neq \emptyset$, say $c(u_7u_4)=4$, then $C(u_7, \{u_1, u_2, u_3\})\subseteq \{1,4\}$, which implies that $\{u_1, u_2, u_3,$ $u_4, u_7, u_{10}\}$ forms a 3-colored $K_6$, a contradiction. Thus $C(u_7, \{u_4, u_5, u_6\})\subseteq \{2, 3\}$. Then $c(u_1u_7)\in\{1,2, 3\}$, which implies that $\{u_1, u_4, u_5, u_6, u_7, u_{10}\}$ forms a 3-colored $K_6$, a contradiction.
\end{proof}

By Claim \ref{cl:B_1} and K\"{o}nig's Theorem, we may assume that $S=\{u_1, u_2\}$, $T=\{u_3, u_4, \ldots,$ $u_9\}$ and $c(u_1u_3)=c(u_2u_4)=1$ without loss of generality. Since $1\notin C(T)$, we may further assume that $c(u_3u_4)=2$. Then $c(u_1u_4), c(u_2u_3)\in \{1,2\}$, so $c(u_1u_2)\notin \{1,2\}$; otherwise $\{u_1, u_2, u_3, u_4, u_{10}\}$ forms a 2-colored $K_5$ and thus $G$ contains a 3-colored $K_6$ by Lemma \ref{le:Gallai}. Without loss of generality, let $c(u_1u_2)=3$. Then $c(u_1u_4)=c(u_2u_3)=1$ in order to avoid a rainbow triangle.

\begin{claim}\label{cl:B_2} For any vertex $u\in T\setminus \{u_3, u_4\}$, we have $4\in C(u, \{u_3, u_4\})\subseteq \{2, 4\}$ or $5\in C(u, \{u_3, u_4\})\subseteq \{2, 5\}$.
\end{claim}

\begin{proof} First, we have $3\notin C(u, \{u_3, u_4\})$; otherwise it is easy to check that $C(\{u, u_1, u_2, u_3, u_4,$ $u_{10}\})=\{1, 2, 3\}$ in order to avoid a rainbow triangle, which is a contradiction. Second, we have $C(u, \{u_3, u_4\})\cap \{4, 5\}\neq \emptyset$; otherwise if $c(uu_3)=c(uu_4)=2$, then we also have $C(\{u, u_1, u_2, u_3, u_4,$ $u_{10}\})=\{1, 2, 3\}$. This contradiction completes the proof of Claim \ref{cl:B_2}.
\end{proof}

By Claim \ref{cl:B_2} and the pigeonhole principle, we may assume that $4\in C(u_i, \{u_3, u_4\})\subseteq \{2, 4\}$ for $i\in \{5, 6, 7\}$.

\begin{claim}\label{cl:B_3} For each $i\in \{5, 6, 7\}$, we have $c(u_i, \{u_3, u_4\})=4$.
\end{claim}

\begin{proof} For a contradiction, suppose that $C(u_i, \{u_3, u_4\})=\{2, 4\}$ for some $i\in \{5, 6, 7\}$, say $c(u_5u_3)=2$ and $c(u_5u_4)=4$. Then $c(u_1u_5)=c(u_2u_5)=1$. If $c(u_6u_3)=4$, then $C(u_6, \{u_1, u_4, u_5\})\subseteq \{1, 2, 4\}$, which implies that $G[\{u_1, u_3, u_4, u_5, u_6, u_{10}\}]$ is a 3-colored $K_6$, a contradiction. Thus $c(u_6u_3)=2$, so $c(u_6u_4)=4$ since $4\in C(u_6, \{u_3, u_4\})$. By symmetry, we have $c(u_7u_3)=2$ and $c(u_7u_4)=4$. Then $c(\{u_1, u_2\}, \{u_6, u_7\})=1$. If $C(\{u_5, u_6, u_7\})\cap \{2, 4\}\neq \emptyset$, then $G[\{u_{10}, u_1, u_3, u_4, \ldots, u_7\}]$ contains a 3-colored $K_6$, a contradiction. Thus $C(\{u_5, u_6, u_7\})\subseteq \{3,5\}$. Then $G[\{u_{10}, u_1, u_2, u_5, u_6, u_7\}]$ is a 3-colored $K_6$, a contradiction.
\end{proof}

By Claim \ref{cl:B_3}, we have $c(\{u_5, u_6, u_7\}, \{u_3, u_4\})=4$. In order to avoid a rainbow triangle, we have $C(\{u_1, u_2\}, \{u_5, u_6, u_7\})\subseteq \{1, 4\}$. If $C(\{u_5, u_6, u_7\})\cap \{2, 4\}\neq \emptyset$ (resp., $3\in C(\{u_5, u_6, u_7\})$), then $G[\{u_{10}, u_1, u_3, u_4, \ldots, u_7\}]$ (resp., $G[\{u_{10}, u_1, u_2, u_4, u_5, u_6, u_7\}]$) contains a 3-colored $K_6$, a contradiction. Thus $c(\{u_5, u_6, u_7\})=5$. Then $G[\{u_{10}, u_{1}, u_4, u_5, u_6, u_7\}]$ is a 3-colored $K_6$. This contradiction completes the proof.

\end{document}